\documentclass[11pt]{article}

\usepackage{iftex}
\ifPDFTeX
  \usepackage[utf8]{inputenc}
  \usepackage[T1]{fontenc}
\else
  \usepackage{fontspec}
  \defaultfontfeatures{Ligatures=TeX}
\fi
\PassOptionsToPackage{colorlinks,urlcolor=blue,citecolor=blue,linkcolor=blue}{hyperref}
\usepackage{xurl}
\usepackage{booktabs}
\usepackage{amsfonts}
\ifdefined\DisableMicrotype
\else
  \ifPDFTeX
    \ifnum\pdfoutput>0
      \usepackage[protrusion=true,expansion=true,tracking=true]{microtype}
    \else
      \usepackage[protrusion=true,expansion=false,tracking=true]{microtype}
    \fi
  \else
    \ifXeTeX
      \usepackage[protrusion=true,expansion=false,tracking=false]{microtype}
    \else
      \usepackage[protrusion=true,expansion=true,tracking=true]{microtype}
    \fi
  \fi
\fi
\usepackage{nicefrac}
\usepackage{bbm}
\usepackage{xspace}
\usepackage{algpseudocode}
\usepackage{algorithm}
\usepackage{bm}
\usepackage{amsmath}
\usepackage{amsthm}
\usepackage{amssymb}
\usepackage{thmtools}
\usepackage{thm-restate}
\usepackage{float}
\usepackage{subcaption}
\usepackage{mathtools}
\usepackage[most]{tcolorbox}
\usepackage{xparse}
\usepackage{paralist}
\usepackage{enumitem}
\usepackage{fancyhdr}
\usepackage{multirow}
\usepackage{pifont}
\usepackage{hyperref}
\usepackage{tikz}
\usepackage[capitalize,nameinlink,noabbrev]{cleveref}
\newif\ifneuripsbuild
\ifdefined\NeuripsBuild
  \neuripsbuildtrue
\else
  \neuripsbuildfalse
\fi
\newcommand{\IfNeuripsBuild}[1]{%
  \ifneuripsbuild
    #1%
  \fi
}
\newcommand{\IfNotNeuripsBuild}[1]{%
  \ifneuripsbuild
  \else
    #1%
  \fi
}
\newif\ifusedeferredproofs
\ifdefined\EnableDeferredProofs
  \usedeferredproofstrue
\else
  \usedeferredproofsfalse
\fi
\ifusedeferredproofs
  \usepackage[
    createShortEnv,
    commandRef=Cref,
    disablePatchSection
  ]{proof-at-the-end}
\fi
\usetikzlibrary{arrows.meta,positioning,shapes.geometric}

\makeatletter
\@ifundefined{theHALG@line}
  {\def\theHALG@line{\thealgorithm.\arabic{ALG@line}}}
  {\renewcommand{\theHALG@line}{\thealgorithm.\arabic{ALG@line}}}
\makeatother

\DeclareMathOperator{\interior}{\mathrm{Int}}
\DeclareMathOperator{\relinterior}{\mathrm{rel.int}}
\DeclareMathOperator*{\argmin}{argmin}

\DeclareMathOperator{\dist}{dist}

\newcommand{\innp}[1]{\left\langle #1 \right\rangle}

\newcommand{\RR}{\mathbb{R}}
\newcommand{\NN}{\mathbb{N}}
\newcommand{\norm}[1]{\left\| #1 \right\|}
\newcommand{\abs}[1]{\left| #1 \right|}

\newcommand{\cC}{\mathcal{C}}
\newcommand{\cM}{\mathcal{M}}

\newcommand{\defeq}{\stackrel{\mathrm{\scriptscriptstyle def}}{=}}

\newcommand{\leanverified}{\xspace\textup{\normalfont\,[\textsc{Lean}]}\xspace}
\newif\ifshowclarify
\showclarifytrue

\newif\ifshowchanges
\showchangestrue
\newcommand{\HideChanges}{\showchangesfalse}

\NewDocumentEnvironment{changedblock}{+b}{%
  \ifshowchanges
    \begingroup\color{blue}#1\endgroup
  \else
    #1
  \fi
}{}

\newif\ifshowresearchboxes
\newif\ifshowmotivationboxes
\newif\ifshowinterpretationboxes
\newif\ifshowverificationboxes

\showresearchboxestrue
\showmotivationboxestrue
\showinterpretationboxestrue
\showverificationboxestrue

\newcommand{\HideAllResearchBoxes}{%
  \showresearchboxesfalse
  \showmotivationboxesfalse
  \showinterpretationboxesfalse
  \showverificationboxesfalse
}

\NewDocumentEnvironment{motivation}{+b}{%
  \ifshowresearchboxes
    \ifshowmotivationboxes
      \begin{tcolorbox}[
        breakable,
        colback=blue!4,
        colframe=blue!45!black,
        title=Motivation,
        fonttitle=\bfseries
      ]
      #1
      \end{tcolorbox}
    \fi
  \fi
}{}

\NewDocumentEnvironment{interpretation}{+b}{%
  \ifshowresearchboxes
    \ifshowinterpretationboxes
      \begin{tcolorbox}[
        breakable,
        colback=green!4,
        colframe=green!35!black,
        title=Interpretation,
        fonttitle=\bfseries
      ]
      #1
      \end{tcolorbox}
    \fi
  \fi
}{}

\NewDocumentEnvironment{verification}{+b}{%
  \ifshowresearchboxes
    \ifshowverificationboxes
      \begin{tcolorbox}[
        breakable,
        colback=gray!5,
        colframe=gray!50,
        title=Verification,
        fonttitle=\bfseries
      ]
      #1
      \end{tcolorbox}
    \fi
  \fi
}{}

\theoremstyle{plain}
\numberwithin{equation}{section}
\newtheorem{theorem}{Theorem}[section]
\numberwithin{theorem}{section}
\newtheorem{corollary}[theorem]{Corollary}

\newtheorem{lemma}[theorem]{Lemma}
\newtheorem{proposition}[theorem]{Proposition}

\theoremstyle{definition}
\newtheorem{definition}[theorem]{Definition}
\newtheorem{remark}[theorem]{Remark}
\newtheorem{example}[theorem]{Example}

\newcounter{deferredproofanchor}[section]

\crefalias{deferredproofanchor}{section}

\newcommand{\IfDeferredProofs}[1]{%
  \ifusedeferredproofs
    #1%
  \fi
}
\newcommand{\DeferredProofNotice}[1]{%
  \par\noindent
  \ifusedeferredproofs
    Proof deferred to \hyperref[#1]{\Cref*{#1} (\nameref*{#1})}.%
  \else
    Proof deferred to \hyperref[#1]{\Cref*{#1}}.%
  \fi
  \par
}
\newcommand{\DeferredProofAnchor}[1]{%
  \refstepcounter{deferredproofanchor}%
  \label{#1}%
}
\ifusedeferredproofs
  \NewDocumentEnvironment{deferredtheorem}{o +b}{%
    \IfNoValueTF{#1}
      {\begin{theoremE}[][proof at the end, restate, category=deferred]}%
      {\begin{theoremE}[#1][proof at the end, restate, category=deferred]}%
    #2%
    \end{theoremE}%
  }{}
  \NewDocumentEnvironment{deferredlemma}{o +b}{%
    \IfNoValueTF{#1}
      {\begin{lemmaE}[][proof at the end, restate, category=deferred]}%
      {\begin{lemmaE}[#1][proof at the end, restate, category=deferred]}%
    #2%
    \end{lemmaE}%
  }{}
  \NewDocumentEnvironment{deferredproposition}{o +b}{%
    \IfNoValueTF{#1}
      {\begin{propositionE}[][proof at the end, restate, category=deferred]}%
      {\begin{propositionE}[#1][proof at the end, restate, category=deferred]}%
    #2%
    \end{propositionE}%
  }{}
  \NewDocumentEnvironment{deferredproof}{o +b}{%
    \IfNoValueF{#1}{\DeferredProofNotice{#1}}%
    \IfNoValueTF{#1}{%
      \begin{textAtEnd}[category=deferred]
\begin{proof}
#2
\end{proof}
      \end{textAtEnd}
    }{%
      \eraseIfNeeded{deferred}%
      \appendtofile{\pratendGeneratePrefixFile{\jobname}deferred.tex}{%
\string\makeatletter\string\Hy@SaveLastskip%
\string\phantomsection\string\label{#1}%
\string\Hy@RestoreLastskip\string\makeatother%
\string\begin{proof}%
\detokenize{#2}%
\string\end{proof}}%
    }%
  }{}
  \NewDocumentEnvironment{deferredproofsubsection}{m m +b}{%
    \DeferredProofNotice{#1}%
    \begin{textAtEnd}[category=deferredproofsections]
\subsection{#2}
\label{#1}
\begin{proof}
#3
\end{proof}
    \end{textAtEnd}
  }{}
  \newcommand{\PrintDeferredProofs}{%
    \IfFileExists{\pratendGeneratePrefixFile{\jobname}deferred.tex}{\printProofs[deferred]}{}%
  }
  \newcommand{\PrintDeferredProofSections}{%
    \IfFileExists{\pratendGeneratePrefixFile{\jobname}deferredproofsections.tex}{\printProofs[deferredproofsections]}{}%
  }
\else
  \NewDocumentEnvironment{deferredtheorem}{o +b}{%
    \IfNoValueTF{#1}{\begin{theorem}}{\begin{theorem}[#1]}%
    #2%
    \end{theorem}%
  }{}
  \NewDocumentEnvironment{deferredlemma}{o +b}{%
    \IfNoValueTF{#1}{\begin{lemma}}{\begin{lemma}[#1]}%
    #2%
    \end{lemma}%
  }{}
  \NewDocumentEnvironment{deferredproposition}{o +b}{%
    \IfNoValueTF{#1}{\begin{proposition}}{\begin{proposition}[#1]}%
    #2%
    \end{proposition}%
  }{}
  \NewDocumentEnvironment{deferredproof}{o +b}{%
    \IfNoValueTF{#1}{%
      \begin{proof}
        #2%
      \end{proof}
    }{%
      \begin{proof}
        \DeferredProofAnchor{#1}%
        #2%
      \end{proof}
    }%
  }{}
  \NewDocumentEnvironment{deferredproofsubsection}{m m +b}{%
    \subsection{#2}\label{#1}
    \begin{proof}
      #3%
    \end{proof}
  }{}
  \newcommand{\PrintDeferredProofs}{}
  \newcommand{\PrintDeferredProofSections}{}
\fi

\graphicspath{{images/}}
\makeatletter
\def\mathcolor#1#{\@mathcolor{#1}}
\def\@mathcolor#1#2#3{%
  \protect\leavevmode
  \begingroup
    \color#1{#2}#3%
  \endgroup
}
\makeatother

\makeatletter
\newenvironment{proof*}[1][\proofname]{\par
  \pushQED{\qed}%
  \normalfont \partopsep=\z@skip \topsep=\z@skip
  \trivlist
  \item[\hskip\labelsep
        \itshape
    #1\@addpunct{.}]\ignorespaces
}{%
  \popQED\endtrivlist\@endpefalse
}
\makeatother

\makeatletter
\AtBeginDocument{%

  \@ifpackageloaded{cleveref}{%
    \crefname{theorem}{Theorem}{Theorems}
    \Crefname{theorem}{Theorem}{Theorems}
    \crefname{corollary}{Corollary}{Corollaries}
    \Crefname{corollary}{Corollary}{Corollaries}
    \crefname{conjecture}{Conjecture}{Conjectures}
    \Crefname{conjecture}{Conjecture}{Conjectures}
    \crefname{assumption}{Assumption}{Assumptions}
    \Crefname{assumption}{Assumption}{Assumptions}
    \crefname{lemma}{Lemma}{Lemmas}
    \Crefname{lemma}{Lemma}{Lemmas}
    \crefname{proposition}{Proposition}{Propositions}
    \Crefname{proposition}{Proposition}{Propositions}
    \crefname{claim}{Claim}{Claims}
    \Crefname{claim}{Claim}{Claims}
    \crefname{fact}{Fact}{Facts}
    \Crefname{fact}{Fact}{Facts}
    \crefname{definition}{Definition}{Definitions}
    \Crefname{definition}{Definition}{Definitions}
    \crefname{remark}{Remark}{Remarks}
    \Crefname{remark}{Remark}{Remarks}
    \crefname{example}{Example}{Examples}
    \Crefname{example}{Example}{Examples}
    \crefname{observation}{Observation}{Observations}
    \Crefname{observation}{Observation}{Observations}
    \crefname{maintheorem}{Main Theorem}{Main Theorems}
    \Crefname{maintheorem}{Main Theorem}{Main Theorems}
  }{}
}
\makeatother

 \HideAllResearchBoxes
\HideChanges

\usepackage[square,comma,sort&compress]{natbib}
\usepackage[margin=1in]{geometry}
\ifPDFTeX
  \usepackage{charter}
\fi
\usepackage{tabularx}

\title{Frank--Wolfe Beyond $1/t$ Convergence}

\author{Sebastian Pokutta\\
Institute of Mathematics, Technische Universit\"at Berlin and\\
Zuse Institute Berlin, Germany\\
\texttt{pokutta@zib.de}}

\date{\today}

\begin{document}

\maketitle

\begin{abstract}
We study the vanilla Frank--Wolfe algorithm for smooth convex minimization over compact convex sets. Well-known lower bounds establish a uniform worst-case $\Omega(1/t)$ primal-gap barrier in the general smooth convex case, and faster convergence usually requires favorable function properties such as Hölder error bounds or strong convexity. We present a new \emph{Local Dual Sharpness (LDS)} condition on the feasible region and the LMO used by the algorithm. For every fixed smooth convex instance satisfying LDS, Frank--Wolfe satisfies $f(x_t)-f^\star=o(1/t)$ under short steps, exact line search, and the classical open-loop rule. This excludes an asymptotic $\Omega(1/t)$ trajectory lower bound on any one such instance; this is not to be confused with a uniform minimax statement. The condition generalizes and localizes uniform convexity of sets and is satisfied globally by every uniformly convex set. To our knowledge, this gives the first fixed-instance $o(1/t)$ result on uniformly convex sets without an objective growth, gradient-lower-bound, or strong-convexity assumption. Combining LDS with a local Hölder error bound quantifies the rate.
\end{abstract}

\section{Introduction}

We consider optimization problems of the form $\min_{x \in \cC} f(x)$, where $f$ is smooth and convex and $\cC$ is a compact convex set. The Frank--Wolfe algorithm \citep{frank1956algorithm}, also called the conditional gradient method \citep{levitin1966constrained}, is one
of the classical projection-free methods for smooth constrained convex
optimization; see \citet{pokutta2023short,
braun2025cgm} for an overview. Its generic convergence rate on a compact convex set is $O(1/t)$, where the lower-bound instance arises from minimizing $f(x) = \norm{x}^2_2$ over the probability simplex \citep{lan2013complexity,jaggi2013revisiting}. Higher convergence rates typically require assumptions on the structure of $f$, e.g., strong convexity or lower-bounded gradients, but it is \emph{a priori} not clear which assumptions are necessary. For example, only recently in \citet{halbey2026lower} it was shown that in the case of strongly convex functions over smooth strongly convex sets the worst-case rate can be as bad as $\Omega(1/t^2)$, answering a long-standing open question arising from \citet{garber2015faster}; shortly afterwards it was shown that for nonsmooth strongly convex sets this even holds for any LMO-based method, i.e., those that access the feasible region solely through linear minimization \citep{grimmer2026uniform}. The overall convergence landscape is complex for the Frank--Wolfe algorithm and we discuss it in detail in the related work section.

In this work we introduce the \emph{Local Dual Sharpness (LDS)} condition, imposed on the feasible region together with its Linear Minimization Oracle (LMO) around a nonempty reference set $M$. The definition and our conditional theorems allow a general $M$; in the principal instance-level results we take $M=\cM=\argmin_{x\in\cC}f(x)$. Thus, for locally curved sets, applicability depends on where the objective places its minimizers; we discuss the precise assumptions below. Once LDS holds around $\cM$, Frank--Wolfe with exact line search, short steps, or the classical open-loop step-size satisfies $f(x_t)-f^\star=o(1/t)$ for that fixed smooth convex instance, without any additional objective growth, strong-convexity, or gradient-lower-bound assumption. For globally uniformly convex sets LDS holds around every reference set, and hence the conclusion applies to every smooth convex objective. The improvement is qualitative and instance-dependent. In particular, we do not claim a uniform burn-in or minimax modulus. Quantitative rates are recovered below by adding a local Hölder error bound.

\IfNotNeuripsBuild{%
\begin{table}[!ht]
\footnotesize
\centering
\setlength{\tabcolsep}{4pt}
\renewcommand{\arraystretch}{1.12}
\begin{tabularx}{\textwidth}{@{}>{\raggedright\arraybackslash}p{1.25cm}>{\raggedright\arraybackslash}p{1.9cm}>{\raggedright\arraybackslash}X>{\raggedright\arraybackslash}p{3cm}>{\raggedright\arraybackslash}p{3.00cm}c@{}}
\toprule
$f$ & $\cC$ & minimizer / regime & rule & bound & key \\
\midrule
\multicolumn{6}{@{}l}{\emph{\textbf{Classical compact-convex regime}}}\\
convex & convex & unrestricted & OL / SS / LS & $O(1/t)$ & [A] \\
convex & convex & unrestricted & FW / FO-LMO & $\Omega(1/t)$ & [B] \\
\addlinespace[0.2em]
\multicolumn{6}{@{}l}{\emph{\textbf{Location-based acceleration}}}\\
SC & convex & $x^\star \in \interior(\cC)$ & SS / LS & $O(e^{-rt})$ & [C] \\
convex & UC$(q)$ & LBG, $\inf_{x\in\cC}\norm{\nabla f(x)} > 0$ & SS / LS & \(O(e^{-rt})\) for \(q=2\); \(O(t^{-q/(q-2)})\) for \(q>2\) & [D] \\
convex & UC$(q)$ & LBG, $\inf_{x\in\cC}\norm{\nabla f(x)} > 0$ & OL \(\ell/(t+\ell)\) & \(O(t^{-\ell})\) for \(q=2\); \(O(t^{-\ell+\varepsilon}+t^{-q/(q-2)})\) for \(q>2\) & [E] \\
\addlinespace[0.2em]
\multicolumn{6}{@{}l}{\emph{\textbf{Curved sets without location information}}}\\
HEB\((\theta)\) & UC$(q)$ & unrestricted & SS / LS & \(O(t^{-1/(1-2\theta/q)})\) & [F] \\
HEB\((\theta)\) & UC$(q)$ & unrestricted & OL \(\ell/(t+\ell)\) & \(O(t^{-\ell+\varepsilon}+t^{-1/(1-2\theta/q)})\) & [G] \\
SC & SC & unrestricted & SS / LS & \(O(1/t^2)\) & [H] \\
SC & SC & unrestricted & SS / LS & \(\Omega(1/t^2)\) & [I] \\
SC & SC & unrestricted & deterministic FO-LMO & \(\Omega(1/t^2)\) & [J] \\
SC quad. & \(\ell_p\)-ball, \(p\ge 3\) & zero-gradient boundary minimizer & SS / LS & \(\Theta(t^{-p/(p-1)})\) & [K] \\
\addlinespace[0.2em]
\multicolumn{6}{@{}l}{\emph{\textbf{Wolfe's proper-face regime}}}\\
SC & polytope & $x^\star \in \relinterior(F)$, \(\dim F \ge 1\) & SS / LS & \(\Omega(1/(t\log^{2+\delta}t))\) i.o. for every \(\delta>0\) & [L] \\
HEB\((\theta)\) & polytope & $x^\star \in \relinterior(F)$, active face identified & OL \(\ell/(t+\ell)\) & \(O(t^{-1/(1-\theta)})\) & [M] \\
\addlinespace[0.2em]
\multicolumn{6}{@{}l}{\emph{\textbf{Beyond-\(1/t\) via local dual sharpness (this paper)}}}\\
convex & convex & LDS, unrestricted & SS / LS & pointwise \(o(1/t)\) &  \\
convex & convex & LDS$(q)$, unrestricted & OL \(\ell/(t+\ell)\), \(\ell\ge2\) & pointwise \(o(1/t)\) &  \\
HEB\((\theta)\) & convex & LDS$(q)$, unrestricted & SS / LS & \(O(t^{-1/(1-2\theta/q)})\) &  \\
HEB\((\theta)\) & convex & LDS$(q)$, unrestricted & OL \(2/(t+2)\) & \(O(t^{-1/(1-2\theta/q)})\) & \\
\bottomrule
\end{tabularx}
\caption{Known convergence-rate landscape for vanilla Frank--Wolfe. All objective classes in the \(f\)-column are assumed convex and smooth, and all feasible sets in the \(\cC\)-column are assumed compact; ``SC'' means strongly convex, ``SC quad.'' means strongly convex quadratic \(f\), UC$(q)$ means power-type \(q\)-uniform convexity, HEB\((\theta)\) denotes a (possibly local) H\"olderian error bound with exponent \(0<\theta\le 1/2\) (manuscript uses \(r=1/\theta\)), and LDS denotes local dual sharpness, with LDS$(q)$ indicating the corresponding power \(q\). SS is the global short-step rule, LS is exact line search, OL denotes the displayed open-loop rule or family, and FO-LMO denotes a deterministic first-order method with one linear-minimization-oracle call per iteration. The LDS little-$o$ rows are fixed-instance asymptotic statements. Rows [B], [I], and [J] are uniform worst-case or finite-horizon lower bounds; the remaining rows use the quantifiers stated in their entries and the accompanying discussion. For respective references see keys in \Cref{rem:current-sota-long}.}
\IfNotNeuripsBuild{\label{tab:known-fw-rates}}
\label{tab:known-fw-rates-long}
\end{table}

\subsection{Related Work and State of the Art\IfNeuripsBuild{ (Comprehensive Version)}}
\IfNotNeuripsBuild{\label{rem:current-sota}}
\label{rem:current-sota-long}

\Cref{tab:known-fw-rates-long} summarizes the vanilla Frank--Wolfe rate landscape most
relevant to the present paper. The organizing principle is the mechanism that
either improves or obstructs the classical \(O(1/t)\) rate. We also separate
large-scale oracle lower bounds from smooth low-dimensional witnesses whenever
both are available.

\paragraph{Classical compact-convex regime.}
On a general compact convex set, vanilla Frank--Wolfe with exact line search,
global short steps, or the classical open-loop rule has the familiar
\(O(1/t)\) primal-gap guarantee [A]. The line-search analysis goes back to
\citet{frank1956algorithm}, while \citet{levitin1966constrained} and
\citet{jaggi2013revisiting} give the modern smooth-convex formulation. This
baseline is sharp in general: simplex-type examples in \cite{jaggi2013revisiting} already show an
\(\Omega(1/t)\) obstruction in the low-iteration regime, and
\citet{lan2013complexity} extends the lower-bound picture to deterministic
first-order methods that access the feasible region through one LMO call per
iteration [B].

\paragraph{Location-based acceleration.}
A first way to beat \(1/t\) is to use information about the position of the
minimizer. If the minimizer lies in the interior of the feasible region, then
short step and exact line search become linear [C], in a line of work going back
to \citet[\S 8]{wolfe70} and \citet{gm86}. A second mechanism is the
lower-bounded-gradient (LBG) regime on curved sets. When
\(\inf_{x\in\cC}\norm{\nabla f(x)}>0\) and \(\cC\) is \(q\)-uniformly convex,
the Frank--Wolfe gap controls the displacement strongly enough to give linear
convergence for \(q=2\) and faster-than-\(1/t\) polynomial decay for \(q>2\).
This picture is classical for strongly convex sets
\citep{levitin1966constrained,demyanov1970approximate,dunn1979rates}, is
presented in modern form by \citet{garber2015faster}, and is extended to
uniformly convex sets and affine-invariant growth frameworks by
\citet{kerdreux2021projection,pena2023affine} [D]. For open-loop rules
\(\gamma_t=\ell/(t+\ell)\), the same mechanism also yields accelerated rates;
the current state of the art is the affine-invariant treatment of
\citet{WPP2023}, building on the earlier \(\eta_t=4/(t+4)\) analysis of
\citet{WKP2022} [E]. Note that rows [C] and [D] are not exhaustive, as they leave out the boundary case in which the unconstrained minimizer lies exactly on $\partial\cC$, so that $\nabla f(x^\star)=0$. In fact, it is this boundary zero-gradient regime that is used by the hard SC/SC instances in [I] and [J].

\paragraph{Curved sets without location information.}
A different family of results does not assume either an interior minimizer or a
gradient bounded away from zero. Here the improvement comes from combining
curvature of the feasible region with growth on the objective. In particular,
\citet{kerdreux2021projection} [F] and \citet{pena2023affine} [G] show that a
H\"olderian error bound together with uniform convexity of the set yields
explicit polynomial rates for short step, line search, and open loop.
Strongly convex objectives over strongly convex sets fit into this picture and
give the familiar \(O(1/t^2)\) upper bound of \citet{garber2015faster} [H]. Recent
lower bounds show that this benchmark is essentially sharp in two complementary
senses: \citet{halbey2026lower} [I] obtain matching smooth small-scale lower bounds
on Euclidean balls and ellipsoids, while \citet{grimmer2026uniform} [J] prove
large-scale lower bounds for deterministic LMO methods. In our notation, the
SC/SC regime is a special case of LDS\((2)\) + HEB\((1/2)\) (equivalently
\(r=2\)), so [I] and [J] match the \(t^{-2}\) exponent obtained by specializing
\Cref{thm:local-gap-heb-rate} in worst-case complexity, although [I] allows the
hard initialization and [J] the hard problem instance and dimension to depend
on the prescribed iteration horizon. At the low-dimensional end, the explicit quadratic over an \(\ell_p\)-ball of
\citet{zimmer2026agentic} [K] exhibits the slower exponent
\(\Theta(t^{-p/(p-1)})\) for short step and exact line search at a
zero-gradient boundary minimizer.

\paragraph{Wolfe's proper-face regime.}
When the minimizer lies in the relative interior of a proper face of a
polytope, the behavior changes again. Wolfe already observed that short step
and exact line search can become much slower in that regime and \citet{Canon_FWbound68} later refined this observation. A precise formulation
of their result is given in \citet[Theorem~2.8]{braun2025cgm}: once an iterate
\(x_{t_0}\notin F\) with
\(f(x_{t_0})<\min_{v\in\operatorname{Vert}(F)}f(v)\) has been reached, for every
\(\delta>0\), the primal gap is \(\Omega(1/(t\log^{2+\delta}t))\) for infinitely
many \(t\) [L]. Consequently, no \(O(1/(t\log^{2+\eta}t))\) guarantee, with
\(\eta>0\), can hold uniformly over this proper-face regime.
Open-loop rules behave differently here: once the active face has been
identified, the affine-invariant weak-growth framework of \citet{WPP2023} [M]
gives \(O(t^{-1/(1-\theta)})\) rates, and for strongly convex objectives this
recovers an \(O(t^{-2})\) decay.

All results in this related work section concern the vanilla Frank--Wolfe algorithm of \citep{frank1956algorithm,levitin1966constrained}. Restarted \citep{kerdreux2018restarting},
away-step/pairwise \citep{lacoste2015global}, blended variants \citep{braun2019blended,tsuji2022pairwise} or otherwise modified variants can exploit additional structural assumptions and yield further acceleration, but those methods are complementary
to the present discussion.
 }

\IfNeuripsBuild{%
\begin{table}[t]
\small
\centering
\setlength{\tabcolsep}{4pt}
\renewcommand{\arraystretch}{1.18}
\begin{tabularx}{\textwidth}{@{}>{\raggedright\arraybackslash\bfseries}p{3.3cm}>{\raggedright\arraybackslash}X@{}}
\toprule
Regime & Representative rate picture \\
\midrule
Compact-convex baseline & Generic \(O(1/t)\) convergence, with matching \(\Omega(1/t)\) lower bounds showing that no uniform improvement is possible without additional structure. \\
\addlinespace[0.2em]
Location-based acceleration & Interior minimizers yield linear short-step / line-search rates. On UC$(q)$ sets with lower-bounded gradient, short step, line search, and open loop become linear or faster polynomial. \\
\addlinespace[0.2em]
Curved sets plus growth & HEB together with UC$(q)$ yields explicit polynomial improvements. In the SC / SC case this recovers \(O(1/t^2)\), and both small-scale and large-scale lower bounds are now known. \\
\addlinespace[0.2em]
Proper-face regime & On polytopes, short step and exact line search can be nearly \(1/t\), whereas open loop can still accelerate after active-face identification. \\
\addlinespace[0.2em]
This paper: LDS & For every fixed instance satisfying LDS around its minimizer set, the primal gap is \(o(1/t)\). LDS$(q)$ plus a local HEB yields explicit polynomial tails. \\
\bottomrule
\end{tabularx}
\caption{Compressed overview of the Frank--Wolfe rate landscape most relevant here. \Cref{rem:current-sota-long,tab:known-fw-rates-long} give the full row-by-row comparison and source discussion.}
\label{tab:known-fw-rates}
\end{table}

\subsection*{Related Work and State of the Art}
\label{rem:current-sota}

The classical compact-convex story is by now well understood: vanilla
Frank--Wolfe has an \(O(1/t)\) upper bound, and matching \(\Omega(1/t)\)
lower bounds follow from simplex-type constructions and deterministic
first-order oracle arguments
\citep{frank1956algorithm,levitin1966constrained,jaggi2013revisiting,lan2013complexity}.
Faster rates traditionally require additional structure. The most classical
examples are interior minimizers and the lower-bounded-gradient regime on
curved sets, which lead to linear or accelerated polynomial rates under short
step, exact line search, and open loop
\citep{wolfe70,gm86,garber2015faster,kerdreux2021projection,WKP2022,WPP2023}.

A complementary body of work shows that curvature of the feasible region alone
does not settle the picture. Uniformly convex sets combined with growth or
error-bound assumptions yield explicit polynomial improvements
\citep{kerdreux2021projection,pena2023affine}, while recent uniform finite-horizon lower bounds show
that even strongly convex objectives over strongly convex sets can be as slow
as \(\Omega(1/t^2)\) in the worst case \citep{halbey2026lower,grimmer2026uniform}. The interior-minimizer and lower-bounded-gradient linear regimes do not cover the boundary zero-gradient case used by these hard instances. In
another direction, proper-face phenomena on polytopes can slow short step and
exact line search to nearly \(1/t\), whereas open-loop rules can still exhibit accelerated convergence once the active face has been identified \citep{wolfe70,Canon_FWbound68,WPP2023}.
The construction in \citet{zimmer2026agentic} provides a complementary low-dimensional worst-case instance for uniformly convex sets and quadratic objectives.

A comprehensive overview and comparison had to be relegated to \Cref{rem:current-sota-long}.
 }

\subsection*{Contribution} Once LDS is assumed around the objective-dependent minimizer set, the convergence result requires no additional \emph{growth or curvature} condition on $f$ beyond smoothness and convexity. For uniformly convex feasible sets, which satisfy LDS around every reference set, this qualification is automatic.

\paragraph{Local Dual Sharpness.} We introduce the \emph{Local Dual Sharpness (LDS)} condition, which subsumes uniform convexity of sets and also applies to suitable locally curved regions. For every fixed smooth convex instance satisfying LDS around its minimizer set, Frank--Wolfe converges as $o(1/t)$ without an objective growth condition. Equivalently, no such fixed instance has a persistent asymptotic $\Omega(1/t)$ primal-gap trajectory. This does not imply that $t\sup_I F_t(I) \to 0$ (with $F_t(I)$ denoting the primal gap of instance $I$ after $t$ steps) over a class whose LDS constants, objective, geometry, or burn-in vary with $I$, i.e., it is not a minimax statement. To the best of our knowledge, the uniformly-convex specialization is the first fixed-instance $o(1/t)$ result for every smooth convex objective on such a set without a Hölder error bound, strong convexity, or a lower-bounded gradient. Prior accelerated results over uniformly convex sets assumed one of these additional objective conditions.

Our local condition should be contrasted with the local scaling perspective in
\citep{dunn1979rates,kerdreux2021localglobal}, where the locality is attached
to a boundary point of $\cC$ and a specific normal direction and generally no acceleration is obtained without additional assumptions on the function. LDS is instead imposed around an entire reference set and uniformly over all LMO directions; when the LMO minimizer is nonunique, its atom-selection rule is fixed as specified below.

\paragraph{Robustness under Step-size Choices.} We prove our results for the three common step-size strategies: short steps arising from optimizing over the smoothness inequality and (exact) line search in \Cref{cor:any-smooth-ss-ls}; for open-loop step-sizes, \Cref{thm:ol-set} considers the classical open-loop \(\gamma_t=2/(t+2)\) and \Cref{rem:ol-family} treats the generalized family \(\gamma_t=\ell/(t+\ell)\). The open-loop proof is significantly more involved, since Frank--Wolfe need not be a descent method in that case.

\paragraph{Quantitative Rates.} Finally we combine LDS with a downstream local Hölder error bound to quantify our $o(1/t)$ rates, yielding new accelerated convergence regimes for Frank--Wolfe. While the obtained accelerated regimes are broader than previously known ones, this quantification comes at a cost of strengthened assumptions for $f$. The obtained quantitative bounds are tail estimates as customary (see \citep{braun2025cgm}): after a finite burn-in \(t_0\) the stated rate holds for all \(t\ge t_0\), and we do not optimize this entry time. In particular, we obtain accelerated sublinear rates of the form \(O(t^{-rq/(rq-2)})\), where \(r\) is the local Hölder error bound order and \(q\) is the LDS power. In \Cref{sec:superflat-short-step} we complement these quantitative rates by showing that $o(1/t)$ rates obtained via LDS can be slower than any polynomial rate of form $O(t^{-1-\eta})$ for any $\eta>0$.

\IfNotNeuripsBuild{%
\bigskip
}

We formally verified the principal arguments in Lean~4 (see e.g.,
\citep{moura2021lean}). A \leanverified\ badge denotes either an exact named
endpoint or an explicitly identified collection of verified ingredients
supporting the displayed statement. For \Cref{cor:any-smooth-ss-ls}, the
zero-gradient and lower-bounded-gradient branches are formalized separately
and combined by the elementary compactness argument shown in the paper. We do not formalize either the standard preprocessing step of passing to
orthonormal coordinates on $\operatorname{aff}(\cC)$ or the concrete geometry
of the illustrative sets. The verification will be made
available on GitHub, and reusable components will be proposed upstream to
mathlib. \IfNeuripsBuild{Some proofs are in the appendix due to space; the Lean
files are supplementary material.}

\section{Preliminaries}\label{sec:preliminaries}

We briefly summarize the notions that we will use throughout and we refer the reader to \citet{braun2025cgm} for an in-depth treatment. We consider the constrained convex optimization problem $\min_{x\in\cC} f(x)$, where $\cC \subseteq \RR^d$ is compact and convex and $f\colon \RR^d\to\RR$ is convex and continuously differentiable. For exposition we work in coordinates on the affine hull of \(\cC\); equivalently, the definitions below may be read relative to \(\operatorname{aff}(\cC)\), with directions and gradient norms projected to its linear span. We define $f^\star \defeq \min_{x\in\cC} f(x)$, $\cM \defeq \argmin_{x\in\cC} f(x)$, and $D \defeq \max_{x,y\in\cC}\norm{x-y}$,
where $f^\star$ is the \emph{optimal objective value}, $\cM$ is the \emph{set of minimizers}, $D$ is the \emph{diameter} of $\cC$, and $\norm{\cdot}$ is the Euclidean norm. Throughout, we assume that $f$ is $L$-smooth on $\cC$ for some \(L>0\), i.e.,
\begin{equation}\label{eq:smoothness}
  f(y)
  \le
  f(x) + \innp{\nabla f(x), y-x} + \frac{L}{2}\norm{y-x}^2
  \qquad \forall x,y\in\cC.
\end{equation}
Given an iterate $x_t\in\cC$, let
$  s_t \in \argmin_{s\in\cC} \innp{\nabla f(x_t), s}$
be a \emph{Frank--Wolfe atom} (also called \emph{Frank--Wolfe vertex} in the case of polytopes) that is returned by a \emph{Linear Minimization Oracle (LMO)} for $\cC$. If the argmin set is not a singleton, we fix a \emph{deterministic single-valued selector} $S(g)\in\argmin_{s\in\cC}\innp{g,s}$ and set $s_t=S(\nabla f(x_t))$. Note that the LDS property below is defined with respect to a given fixed selector $S$; in particular LDS for one tie-breaking rule does not necessarily imply LDS for another rule. We write $F_t \defeq f(x_t)-f^\star = f(x_t)-f(x^\star)$ as the \emph{primal gap}, where $x^\star$ is any optimal solution, $g_t \defeq \innp{\nabla f(x_t), x_t-s_t}$ as the \emph{Frank--Wolfe gap}, and $d_t \defeq \norm{x_t-s_t}$. By convexity we immediately obtain $0 \le F_t \le g_t$.

We will be concerned with the Frank--Wolfe algorithm \citep{frank1956algorithm,levitin1966constrained}, stated in \cref{alg:fw}. The key point is that it is \emph{projection-free} by forming its updates as convex combinations of the form $x_{t+1} = x_t + \gamma_t (s_t-x_t)$.

For the choice of the step-size rule in Line~\ref{eq:stepsize} of \cref{alg:fw}, we will consider the following three rules:
\begin{align}
  \gamma_t^{\mathrm{ss}}
  &\defeq
  \min\left\{\frac{g_t}{L d_t^2},1\right\}, & \textit{(short steps)}
  \label{eq:short-step}
  \\
  \gamma_t^{\mathrm{ls}}
  &\in
  \argmin_{\gamma\in[0,1]} f\bigl(x_t+\gamma(s_t-x_t)\bigr), & \textit{(exact line search)}
  \label{eq:line-search}
  \\
  \gamma_t^{\mathrm{ol}}
  &\defeq \frac{2}{t+2}. & \textit{(open-loop)}
  \label{eq:open-loop}
\end{align}
We will later also consider generalized open-loop step-sizes of the form $\gamma_t \defeq \frac{\ell}{t+\ell}$ with $\ell \in \NN$.

The following properties and results are classical; we refer the reader to \citet{braun2025cgm}.

\begin{algorithm}
  \caption{The Frank--Wolfe algorithm \citep{frank1956algorithm}}\label{alg:fw}
  \begin{algorithmic}[1]
    \Require $x_0\in\cC$ and a rule for choosing $\gamma_t \in [0,1]$
    \For{$t=0,1,2,\dots$}
      \State $s_t \gets \argmin_{s\in\cC} \innp{\nabla f(x_t), s}$
      \State \label{eq:stepsize} choose $\gamma_t$ according to
        \eqref{eq:short-step}, \eqref{eq:line-search}, or \eqref{eq:open-loop}
      \State $x_{t+1} \gets x_t + \gamma_t(s_t-x_t)$
    \EndFor
  \end{algorithmic}
\end{algorithm}

\begin{proposition}[Primal progress estimate\leanverified]
\label{prop:standard-descent}
Let $x_{t+1}=x_t+\gamma_t(s_t-x_t)$ with $\gamma_t\in[0,1]$.
Then   $F_{t+1}
  \le
  F_t - \gamma_t g_t + \frac{L}{2}\gamma_t^2 d_t^2$.
In particular:
\begin{enumerate}[leftmargin=2em]
  \item for the short-step rule and exact line search, we have $    F_{t+1}
    \le
    F_t - \frac12 \min\left\{g_t,\frac{g_t^2}{L d_t^2}\right\};
$
  \item for the open-loop rule, we have $    F_{t+1}
    \le
    \left(1-\frac{2}{t+2}\right)F_t
    + \frac{L}{2}\left(\frac{2}{t+2}\right)^2 d_t^2
$.
\end{enumerate}
\end{proposition}

The next proposition states the standard primal convergence guarantee for the Frank--Wolfe algorithm due to \citet{frank1956algorithm,levitin1966constrained}, from the first iterate onward. The proof is standard; see e.g., \citep{jaggi2013revisiting,pokutta2023short,braun2025cgm}.

\begin{proposition}[$O(1/t)$ primal convergence for the Frank--Wolfe algorithm\leanverified]\label{prop:standard-rate} Let $f$ be an $L$-smooth and convex function for some \(L>0\), and let $\cC$ be a non-empty compact convex set. Running \cref{alg:fw}, for each of the three step-size rules \eqref{eq:short-step}, \eqref{eq:line-search}, and \eqref{eq:open-loop}, the iterates $x_t$ satisfy
\begin{equation}
  F_t \le \frac{2LD^2}{t+2}
  \qquad \forall t\ge 1;
  \qquad \text{in particular } F_t\to 0.
\end{equation}
\end{proposition}

We will also need the next simple observation to deal with possibly non-unique minimizers.

\begin{lemma}[Distance to the minimizer set\leanverified]\label{lem:minset-distance} Under the assumptions from \cref{prop:standard-rate}, further assume $F_t\to 0$ and let $\cM = \{x\in\cC : f(x)=f^\star\}$
be the set of minimizers. Then
\(
  \dist(x_t,\cM)\to 0
\).
\end{lemma}

Finally, we will be concerned with uniformly convex sets as feasible regions.

\begin{definition}[Power-type uniform convexity]\label{def:uc}
Let $q\in\RR$ satisfy $q\ge 2$, and let $\alpha>0$.
We say that $\cC$ is \emph{$(\alpha,q)$-uniformly convex} if for every
$x,y\in\cC$, every $\lambda\in[0,1]$, and every $z\in\RR^d$ with
$\norm{z}=1$,
\begin{equation}
  \lambda x + (1-\lambda)y
  + \lambda(1-\lambda)\alpha \norm{x-y}^q z
  \in \cC.
\end{equation}
\end{definition}

The case $q=2$ is exactly the usual \emph{strong convexity} of the feasible set.
Thus every strongly convex set is uniformly convex, while the definition also
includes genuinely higher-order curved bodies such as $\ell_p$ balls for
$p>2$~\citep{kerdreux2021projection}. The next lemma turns contractions into rates.

\begin{deferredlemma}[Discrete power-law descent\leanverified]\label{lem:power-descent}
Let $(a_t)$ be a nonnegative sequence.
Assume that for some \(0<\eta\le 1\), \(0<r\le 1\), and some \(t_0\in\NN\)
with \(a_{t_0}\le 1\), it holds $a_{t+1} \le a_t - \eta a_t^{1+r}$ for all $t\ge t_0$.
Then $a_t = O\left(t^{-1/r}\right)$.
\end{deferredlemma}
\begin{deferredproof}[app:proof-power-descent]
If \(a_t=0\) for some \(t\ge t_0\), then the nonnegativity and the recursion
force the tail to stay zero, so there is nothing to prove. Otherwise, replacing
\(\eta\) by \(\bar\eta\defeq \min\{\eta,1/2\}\), we have
\(\bar\eta a_t^r\le 1/2\) for all \(t\ge t_0\).
Then
\[
  a_{t+1} \le a_t(1-\bar\eta a_t^r),
  \qquad t\ge t_0.
\]
Hence
\[
  a_{t+1}^{-r}
  \ge
  a_t^{-r}{(1-\bar\eta a_t^r)}^{-r}
  \ge
  a_t^{-r}\bigl(1+r\bar\eta a_t^r\bigr)
  =
  a_t^{-r} + r\bar\eta,
\]
where we used ${\bigl(1-u\bigr)}^{-r}\ge 1+ru$ for $0 \le u < 1$.
Therefore $a_t^{-r}$ grows at least linearly, so $a_t = O(t^{-1/r})$.
\end{deferredproof}

\section{Local Dual Sharpness}
\label{sec:local-gap}

The following definition introduces the \emph{Local Dual Sharpness (LDS)} condition that we will be working with; we write $M$ for a general set and $\cM$ for the set of minimizers.

\begin{definition}[Local dual sharpness around a reference set $M$]
\label{def:local-gap}
Let $M\subseteq\cC$ be nonempty, let $A>0$, and let $q\in\RR$ satisfy $q\ge 2$.
We say that $\cC$ together with its fixed LMO atom-selection rule satisfies
\emph{local dual sharpness around $M$ with constants $(A,q)$} if
there exists $\rho>0$ such that whenever $x\in\cC$ with $\dist(x,M)<\rho$,
then for every $g\in\RR^d$ and the atom
\(s\in\argmin_{y\in\cC}\innp{g,y}\) returned by that fixed LMO selection,
it holds:
\[
  A\norm{g}\,\norm{x-s}^q \le \innp{g,x-s}.
\]
\end{definition}

\begin{remark}[Dependence on the atom-selection rule]
\label{rem:dual-sharpness-selector}
Local dual sharpness is a property of the support map induced by the
atom-selection rule of the LMO, not just of the feasible region \(\cC\).
Indeed, for fixed \(x\) and \(g\), the Frank--Wolfe dual gap
$
  \innp{g,x-s}
  =
  \innp{g,x} - \min_{y\in\cC}\innp{g,y}
$
is the same for every minimizing atom \(s\in\argmin_{y\in\cC}\innp{g,y}\).
What can change across selections though is the displacement \(\norm{x-s}\).
The subtlety disappears when the supporting atom is unique, but if the
exposed face is flat then different selections on that face can produce
different displacements and therefore different admissible constants \((A,q)\),
or even validity for one selection and failure for another. We thus assume that the selection rule is deterministic and fixed to avoid this complication.
\end{remark}

\paragraph{Support-function form.}
Let $S(g)$ be the fixed selector and let $\sigma_{\cC}(u)\defeq\sup_{y\in\cC}\innp{u,y}$ be the support function. Then LDS is equivalently
\[
  \innp{g,x}+\sigma_{\cC}(-g)
  \ge A\norm{g}\,\norm{x-S(g)}^q
\]
near $M$. If the supporting point is unique, then $S(g)=\nabla\sigma_{\cC}(-g)$ and this is a power-type exposure inequality for the support function.

Checking whether LDS holds means certifying the LDS inequality simultaneously for every $x$ in a neighborhood of $M$ and every normalized direction $g$. This paper does not study the computational complexity of that certification problem. Useful \emph{sufficient certificates} include global uniform convexity; a uniform affine-hull interior-ball condition around $M$; the curved-patch plus residual-gap criterion in \Cref{prop:patch-local-gap}; or a direct positive lower bound on the normalized ratio over $\norm{g}=1$. If $M=\cM$ is not known, it must first be localized by problem-specific optimality information. Note, however, that the Frank--Wolfe algorithm is adaptive, i.e., it automatically benefits if LDS holds, not requiring certification.

\begin{remark}[Beyond lower-bounded gradients $\nabla f(\cC) > 0$]
\label{rem:ssls-why-minimizer}
Assume that LDS holds around the minimizer set \(\cM\) with
constants \((A,q)\).
If the gradients were bounded away from zero,
i.e., \(\norm{\nabla f(x_t)}\ge G>0\), then once \(x_t\) is close enough to \(\cM\)
the LDS inequality would give \(A G d_t^q \le g_t\),
so that \(g_t\to 0\) would force \(d_t\to 0\). This would recover the accelerated convergence of \citep{levitin1966constrained,demyanov1970approximate}. Here, however, we consider the more involved general case without such assumptions, so that \(\norm{\nabla f(x_t)}\) may itself tend to zero near \(\cM\).
\end{remark}

\begin{proposition}[Uniform convexity implies local dual sharpness\leanverified]
\label{prop:uc-implies-local-gap}
If $\cC$ is $(\alpha,q)$-uniformly convex, then for every nonempty
$M\subseteq\cC$, any LMO with any fixed atom-selection rule satisfies local dual sharpness
around $M$ with constants $(\alpha/2,q)$.
\end{proposition}

\begin{proof}
By \Cref{lem:uc-gap},
$\frac{\alpha}{2}\norm{g}\,\norm{x-s}^q\le\innp{g,x-s}$
holds for every $x\in\cC$, every $g\in\RR^d$, and every selected minimizing
atom $s$. Thus the inequality holds globally for any $\rho>0$ in
\Cref{def:local-gap}.
\end{proof}

The local dual sharpness condition is more general than (global) uniform convexity; we present two examples in \cref{sec:lds-vs-uc} in the appendix. Its validity around $\cM$, however, depends on the location of that objective-dependent set, e.g., in \cref{fig:stadium-local-gap-regimes} it holds on the rounded caps but not on a flat side. Strict convexity alone also does not imply LDS; see \cref{rem:strict-not-enough}. Our setup differs from the local uniform convexity and local scaling perspective of \citep{dunn1979rates,kerdreux2021localglobal}, where locality is expressed at one boundary point and a singled-out normal direction. Here the condition is imposed on the fixed LMO support map around an entire reference set and over all LMO directions.

\subsection{Short Steps and Exact Line Search}
\label{sec:ss-and-ls}

Here and in the following, we often have to pay special attention to the potential $\abs{\cM} > 1$ case. We make this explicit by splitting theorems.

\begin{theorem}[LDS + short steps or exact line search / conditional form\leanverified]
\label{thm:conditional-local-gap}
Let $M\subseteq\cC$ be nonempty.
Assume that the Frank--Wolfe iterates are generated either by the global
short-step rule or by exact line search.
Suppose that $\cC$ together with its LMO satisfies local dual sharpness
around $M$ with some constants $(A,q)$.
Let $x^\star\in\cC$, and let $(p_t)\subseteq M$ satisfy
$f(p_t)=f(x^\star)$ and $\nabla f(p_t)=0$ for all $t$.
If $\delta_t \defeq \norm{x_t-p_t}\to 0$,
then
\[
  (t+2)\bigl(f(x_t)-f(x^\star)\bigr)\to 0.
\]
\end{theorem}

\begin{proof}
Define $H_t \defeq \frac{1}{F_t}$, where $F_t$ is the primal gap. If $F_t=0$ for some $t$, then the claim is immediate.
So we assume $F_t>0$ for all $t$. Let $\rho>0$ be given by \Cref{def:local-gap}.
Since
$\dist(x_t,M)\le \norm{x_t-p_t}=\delta_t \to 0$,
there exists $t_0$ such that \(\dist(x_t,M)<\rho\) for all \(t\ge t_0\).
Hence for every \(t\ge t_0\), local dual sharpness yields $A\norm{\nabla f(x_t)}\,d_t^q \le g_t$.

Convexity at the zero-gradient minimizer $p_t$ gives
$
  F_t
  =
  f(x_t)-f(p_t)
  \le
  \innp{\nabla f(x_t),x_t-p_t}
  \le
  \norm{\nabla f(x_t)}\,\delta_t
$. Moreover, since \(\nabla f(p_t)=0\) and \(f\) is \(L\)-smooth, we have
$F_t \le \frac{L}{2}\delta_t^2$. Combining the first estimate with the local dual sharpness inequality gives
\begin{equation}
  \label{eq:lds-ss-ls}
  A\frac{F_t}{\delta_t}d_t^q \le g_t
  \qquad \forall t\ge t_0.
\end{equation}

By \Cref{prop:standard-descent}, the short-step and exact-line-search rules
both satisfy
\[
  F_{t+1}
  \leq
  F_t-\frac12\min\left\{g_t,\frac{g_t^2}{L d_t^2}\right\}.
\]
Therefore
\[
  H_{t+1}-H_t
  =
  \frac{F_t-F_{t+1}}{F_t\,F_{t+1}}
  \ge
  \frac{F_t-F_{t+1}}{F_t^2},
\]
where the latter inequality holds because $F_t$ is monotone decreasing. We split into two cases.

\noindent
\emph{Case 1: $g_t \ge L d_t^2$.}
Then $F_t-F_{t+1} \ge \frac12 g_t \ge \frac12 F_t$,
so  $H_{t+1}-H_t \ge \frac{1}{2F_t} \ge \frac{1}{L\delta_t^2}$ follows.

\noindent
\emph{Case 2: $g_t \le L d_t^2$.}
Then $F_t-F_{t+1}
  \ge
  \frac{g_t^2}{2L d_t^2}$. Using \(g_t\ge F_t\) and \eqref{eq:lds-ss-ls} we obtain the two inequalities
\begin{equation}
  \label{eq:local-gap-case-two}
  H_{t+1}-H_t
  \ge
  \frac{g_t^2}{2L d_t^2 F_t^2}
  \ge
  \frac{1}{2L d_t^2},
  \qquad \text{and} \qquad
  H_{t+1}-H_t
  \ge
  \frac{A^2 d_t^{2q-2}}{2L\delta_t^2}.
\end{equation}
Let \(\theta = (q-1)/q\) and since \(\max\{u,v\}\ge u^\theta v^{1-\theta}\) for all \(u,v\ge 0\), \eqref{eq:local-gap-case-two} implies
\[
  H_{t+1}-H_t
  \ge
  \left(\frac{1}{2L d_t^2}\right)^{(q-1)/q}
  \left(\frac{A^2 d_t^{2q-2}}{2L\delta_t^2}\right)^{1/q}
  =
  \frac{A^{2/q}}{2L}\,\delta_t^{-2/q}.
\]

After increasing \(t_0\) if needed, we may also assume \(\delta_t<1\) for all
\(t\ge t_0\). Since \(q\ge 2\), this gives \(\delta_t^{-2}\ge \delta_t^{-2/q}\),
so Case 1 also yields $H_{t+1}-H_t \ge \frac{1}{L}\,\delta_t^{-2/q}$.
Thus both cases imply $H_{t+1}-H_t \ge \min\left\{\frac{1}{L},\frac{A^{2/q}}{2L}\right\} \delta_t^{-2/q}$ for all $t \geq t_0$. Since \(\delta_t\to 0\), the reciprocal increments \(H_{k+1}-H_k\) diverge to
\(+\infty\) along the tail. Their Ces\`aro averages therefore diverge as well,
i.e.,
\[
  \frac{H_t}{t} - \frac{H_0}{t}
  =
  \frac{1}{t} \sum_{k = 0}^{t-1} \left( H_{k+1} - H_k \right) \to +\infty.
\]
Thus \(\frac{H_t}{t}\to \infty\), equivalently \(tF_t\to 0\). In particular
\(H_t\to\infty\), hence \(F_t\to 0\). Therefore
\((t+2)F_t = tF_t + 2F_t \to 0\).
\end{proof}

\begin{theorem}[LDS + short steps or line search with a zero-gradient minimizer\leanverified]
\label{thm:local-gap-zgm}
Let $\cM$ be the minimizer set.
Assume that the Frank--Wolfe iterates are generated either by the global
short-step rule or by exact line search.
Suppose that $\cC$ together with its LMO satisfies local dual sharpness
around $\cM$ with some constants $(A,q)$.
If $\cM$ contains a point $x^\star$ with $\nabla f(x^\star)=0$, then
$f(x_t)-f^\star = o(1/t)$.
\end{theorem}

\begin{proof}
By \Cref{prop:standard-rate}, $F_t\to 0$.
Hence \Cref{lem:minset-distance} gives $\dist(x_t,\cM)\to 0$. Choose $p_t\in\cM$ with
$\norm{x_t-p_t}=\dist(x_t,\cM)$.
Then $\delta_t\defeq \norm{x_t-p_t}\to 0$.
Since $x^\star$ is a zero-gradient point of a differentiable convex function,
it is a global minimizer on the ambient space.
Therefore every point of \(\cM\) is also a zero-gradient minimizer, so
$\nabla f(p_t)=0$ for all $t$. Applying \Cref{thm:conditional-local-gap} yields
$(t+2)\bigl(f(x_t)-f^\star\bigr)\to 0$.
\end{proof}

The proof of the next proposition is similar to \citet{demyanov1970approximate,levitin1966constrained,garber2015faster}; see also \citet{braun2025cgm}.
\begin{deferredproposition}[LDS + short steps or exact line search when $\nabla f(\cC) > 0$\leanverified]
\label{prop:local-gap-polyak}
Let $\cM$ be the minimizer set.
Assume that the Frank--Wolfe iterates are generated either by the global
short-step rule or by exact line search.
Suppose that $\cC$ together with its LMO satisfies local dual sharpness
around $\cM$ with some constants $(A,q)$ and that $G \defeq \min_{x\in\cC}\norm{\nabla f(x)} > 0$. Then:
\begin{enumerate}[leftmargin=2em]
  \item if $q=2$, then with $\kappa \defeq \frac12 \min\left\{1,\frac{AG}{L}\right\}$,
  there exists $t_0\in\NN$ such that $F_t\le (1-\kappa)^{t-t_0} F_{t_0}$ for all $t \ge t_0$.
  \item if $q>2$, then there exist $t_0\in\NN$ and $K>0$ such that $F_t\le \frac{K}{(t-t_0+1)^{q/(q-2)}}$ for all $t \ge t_0$.
\end{enumerate}
\end{deferredproposition}
\begin{deferredproof}[proof:prop:local-gap-polyak]
By \Cref{prop:standard-rate}, $F_t\to 0$.
Hence \Cref{lem:minset-distance} gives $\dist(x_t,\cM)\to 0$.
Let $\rho>0$ be given by \Cref{def:local-gap}, and choose $t_0$ so that
\(\dist(x_t,\cM)<\rho\) for all \(t\ge t_0\). Then
\(A\norm{\nabla f(x_t)}\,d_t^q \le g_t\) and hence
\[
  AG\,d_t^q \le g_t \qquad \forall t\ge t_0.
\]
\emph{Case $q=2$.} Then \(d_t^2 \le g_t/(AG)\), so
\(\frac{g_t^2}{L d_t^2}\ge \frac{AG}{L}g_t\). Therefore
\Cref{prop:standard-descent} and \(F_t\le g_t\) give
\[
  F_{t+1}\le F_t-\kappa g_t \le (1-\kappa)F_t
  \qquad \forall t\ge t_0,
  \qquad
  \kappa \defeq \frac12 \min\left\{1,\frac{AG}{L}\right\},
\]
and iterating yields
\[
  F_t\le (1-\kappa)^{t-t_0} F_{t_0}
  \qquad \forall t\ge t_0.
\]
\emph{Case $q>2$.} Again \Cref{prop:standard-descent} and \(F_t\le g_t\) yield
\[
  F_{t+1}
  \le
  F_t - \frac12 \min\left\{F_t,\frac{(AG)^{2/q}}{L} F_t^{2-2/q}\right\}
  \qquad \forall t\ge t_0.
\]
By \Cref{prop:standard-rate}, after increasing \(t_0\) if needed we may
assume \(F_t\le 1\) for all \(t\ge t_0\); since \(q>2\), this implies
\(F_t \ge F_t^{2-2/q}\). Hence
\[
  F_{t+1}\le F_t-\eta F_t^{2-2/q}
  \qquad \forall t\ge t_0,
  \qquad
  \eta \defeq \frac12 \min\left\{1,\frac{(AG)^{2/q}}{L}\right\},
\]
and \Cref{lem:power-descent} applied to the shifted tail gives a constant
\(K>0\) such that
\[
  F_t\le \frac{K}{(t-t_0+1)^{q/(q-2)}}
  \qquad \forall t\ge t_0.
\]
In either case the tail is asymptotically strictly faster than \(1/t\), hence
\(f(x_t)-f^\star=o(1/t)\).
\end{deferredproof}

Combining \Cref{thm:local-gap-zgm,prop:local-gap-polyak} gives:

\begin{corollary}[Beyond $1/t$ under local dual sharpness for short step and exact line search\leanverified]
\label{cor:any-smooth-ss-ls}
Let $\cC$ be compact and convex, and let $f\colon\RR^d\to\RR$ be smooth and
convex. Suppose that $\cC$ together with its LMO satisfies local dual sharpness
around the minimizer set $\cM$ with some constants $(A,q)$.
Then both the global short-step rule and exact line search satisfy
\[
  f(x_t)-f^\star = o(1/t).
\]
\end{corollary}
\begin{proof}
If for some $x\in\cC$ we have $\nabla f(x)=0$, then $x$ is a global
minimizer because of convexity, so that $x\in\cM$ and \Cref{thm:local-gap-zgm} applies. Otherwise
$\nabla f$ has no zero on the compact set $\cC$. Continuity of the gradient
then yields $\min_{x\in\cC}\norm{\nabla f(x)}>0$, and
\Cref{prop:local-gap-polyak} applies.
\end{proof}

\subsection{Open-Loop Step-Sizes}
\label{sec:open-loop}

We will now consider the open-loop step-size case. For the sake of exposition
we prove the classical \(\frac{2}{t+2}\) variant in the main text. For fixed \(\ell\ge2\), the same qualitative argument
extends to \(\frac{\ell}{t+\ell}\); see
\Cref{rem:ol-family} in the appendix. In contrast to the previous arguments for short steps and line
search in \Cref{sec:ss-and-ls}, we do not have one-step contractions and hence
the argument is significantly more involved; we will also not have to
distinguish cases according to whether \(\nabla f(\cC) > 0\).

\begin{deferredtheorem}[LDS + classical open-loop / conditional form\leanverified]
\label{thm:conditional-ol}
Let $M\subseteq\cC$ be nonempty.
Let $(x_t)$ be the Frank--Wolfe iterates generated by the classical open-loop
rule $\gamma_t=2/(t+2)$.
Suppose that $\cC$ together with its LMO satisfies local dual sharpness
around $M$ with some constants $(A,q)$, where \(q\ge2\).
Let \(x^\star\in\cM\), and let \((p_t)\subseteq M\) satisfy
\(f(p_t)=f(x^\star)\) for all \(t\).
If $\delta_t \defeq \norm{x_t-p_t}\to 0$, then
\[
  (t+2)\bigl(f(x_t)-f(x^\star)\bigr)\to 0.
\]
\end{deferredtheorem}

\begin{deferredproof}[proof:thm:conditional-ol]
For convenience define $h_t \defeq (t+2)F_t$, and observe $0 \le h_t \le 2LD^2$ for all $t\geq 1$ by \Cref{prop:standard-rate}; below we always increase burn-in indices so that \(t_0 \ge 1\). Let $\rho>0$ be given by \Cref{def:local-gap} and for every $t$ with $h_t>0$, set $r_t \defeq \norm{p_t-s_t}$.
Since \(p_t\in M\), one has \(\dist(p_t,M)=0<\rho\).
Hence convexity together with local dual sharpness at \(p_t\) gives
$F_t \le \innp{\nabla f(x_t),x_t-p_t}$ and $A\norm{\nabla f(x_t)}\,r_t^q \leq \innp{\nabla f(x_t),p_t-s_t}$.
Adding up the two inequalities we obtain
\[
  g_t
  =
  \innp{\nabla f(x_t),x_t-s_t}
  \ge
  F_t + A\norm{\nabla f(x_t)}\,r_t^q.
\]
If $F_t>0$ and hence $\delta_t>0$, using  $F_t \le \norm{\nabla f(x_t)}\,\delta_t$
we obtain
\[
  g_t
  \ge
  F_t + A\frac{F_t}{\delta_t}r_t^q.
\]
We consider two cases. \emph{Case 1:} If $q=2$, then
$
  g_t \ge F_t + A\frac{F_t}{\delta_t}r_t^2.
$
Fix $\varepsilon>0$ and define
\[
  a_t\defeq\frac{t(t+3)}{(t+2)^2},
  \qquad
  c_t\defeq\frac{t+3}{(t+2)^2},
  \qquad
  C_{\mathrm s}\defeq\frac{4L}{A}.
\]
Because $\delta_t\to0$ and $c_t\to0$, there exists $t_0\ge1$ such that for
all $t\ge t_0$,
\begin{equation}\label{eq:ol-q2-small}
  C_{\mathrm s}\delta_t<\frac{\varepsilon}{2},
  \qquad
  4L\delta_t^2<\frac{\varepsilon}{2},
  \qquad
  2LD^2c_t<\frac{\varepsilon}{2}.
\end{equation}
Also $d_t\le\delta_t+r_t$, hence $d_t^2\le2\delta_t^2+2r_t^2$.
Substituting the bounds for $g_t$ and $d_t^2$ into
\Cref{prop:standard-descent} yields
\begin{equation}\label{eq:ref-progres-base}
  F_{t+1}
  \le
  (1-\gamma_t)F_t+L\gamma_t^2\delta_t^2
  +\left(L\gamma_t^2-A\gamma_t\frac{F_t}{\delta_t}\right)r_t^2.
\end{equation}
The sign of the coefficient of $r_t^2$ yields two cases. If
$L\gamma_t^2>A\gamma_t F_t/\delta_t$, then
\[
  F_t<\frac{L}{A}\gamma_t\delta_t
  =\frac{2L}{A}\frac{\delta_t}{t+2},
  \qquad\text{so that}\qquad
  h_t<\frac{2L}{A}\delta_t
  \le C_{\mathrm s}\delta_t<\frac{\varepsilon}{2}.
\]
We call this the small branch. Otherwise the coefficient of $r_t^2$ is
nonpositive; discarding that term from \eqref{eq:ref-progres-base} and
multiplying by $t+3$ gives the refined recurrence
\begin{equation}\label{eq:ol-q2-refined}
  h_{t+1}\le a_t h_t+ 4L\delta_t^2 c_t.
\end{equation}
Moreover, independently the standard progress estimate gives
\begin{equation}\label{eq:ol-q2-standard}
  h_{t+1}\le a_t h_t+2LD^2c_t.
\end{equation}
Compared with \eqref{eq:ol-q2-standard}, the refined recurrence replaces the fixed diameter error $D^2$ by the vanishing distance error $\delta_t^2$ and is therefore asymptotically sharper on the tail.

We first prove forward invariance of the sublevel set
$\{h_t\le\varepsilon\}$ for $t\ge t_0$. If the small branch holds, then
\eqref{eq:ol-q2-standard}, $a_t\le1$, and
\eqref{eq:ol-q2-small} give $h_{t+1}<\varepsilon$. Otherwise
\eqref{eq:ol-q2-refined} and $h_t\le\varepsilon$ give
\[
  h_{t+1}
  \le a_t\varepsilon+c_t\frac{\varepsilon}{2}
  \le\varepsilon,
\]
because $a_t+c_t/2\le1$. This establishes the forward invariance.

Next suppose $h_t>\varepsilon$. The small branch is impossible, and hence
\eqref{eq:ol-q2-refined} gives
Since $1-a_t=(t+4)/(t+2)^2$ and $h_t>\varepsilon$, we obtain
\begin{align*}
  h_t-h_{t+1}
  &\ge (1-a_t)h_t-c_t\frac{\varepsilon}{2}\\
  &\ge \frac{\varepsilon(t+5)}{2(t+2)^2}\\
  &\ge \frac{\varepsilon}{2(t+2)}.
\end{align*}

If some arbitrarily late $h_n$ were larger than $\varepsilon$, forward
invariance would force $h_t>\varepsilon$ for every $t_0\le t\le n$.
Summing the last inequality from $t_0$ to $n-1$ would then give
\[
  h_n
  \le h_{t_0}-\frac{\varepsilon}{2}
      \sum_{t=t_0}^{n-1}\frac{1}{t+2},
\]
which is negative for all sufficiently large $n$, contradicting $h_n\ge0$.
Therefore $h_t\le\varepsilon$ eventually. Since $\varepsilon>0$ was
arbitrary, $h_t\to0$.
This proves the theorem when $q=2$.

\emph{Case 2:} We therefore assume $q>2$ below. Fix $\varepsilon>0$ and define
$\beta \defeq \frac{2}{q-2}$, $C_\varepsilon \defeq \frac{16L}{qA\varepsilon}$, and
  $R_\varepsilon \defeq \frac{q-2}{q} L C_\varepsilon^\beta$.
Because $\delta_t\to 0$, both $4L\delta_t^2$ and $4R_\varepsilon \delta_t^\beta$
tend to zero.
Hence there exists $t_0$ such that for all $t\ge t_0$,
\begin{equation}\label{eq:coeff-small}
  4L\delta_t^2 + 4R_\varepsilon\delta_t^\beta < \frac{\varepsilon}{2},
  \qquad \text{ and } \qquad
  2LD^2 \frac{t+3}{(t+2)^2} < \frac{\varepsilon}{2}.
\end{equation}

The key point is a dichotomy that we derive in the following. To this end, observe that in the descent estimate the only unknown quantity on the right-hand side is $r_t$. We isolate its contribution as $\Phi_t(r)
  \defeq
  L\gamma_t^2 r^2 - A\gamma_t \frac{F_t}{\delta_t} r^q$, with $r \geq 0$.
Using
\(
  d_t^2 \le 2\delta_t^2 + 2r_t^2
\)
and
\(
  g_t \ge F_t + A(F_t/\delta_t)r_t^q
\)
in the primal progress estimate from \Cref{prop:standard-descent} gives,
similar to before, $F_{t+1}
  \leq
  (1-\gamma_t)F_t + L\gamma_t^2\delta_t^2 + \Phi_t(r_t)$, except we cannot factor out the $r_t$ as before due to mismatching exponents.
However, it is enough to obtain a uniform upper bound on $\Phi_t(r)$ for $r\ge 0$.
Whenever $F_t>0$, define $\lambda_t \defeq \frac{4L\gamma_t\delta_t}{qA F_t}$, then
\[
  \Phi_t(r)
  =
  L\gamma_t^2\left(r^2-\frac{4}{q\lambda_t}r^q\right).
\]
Applying Young's inequality to
\(
  r^2=\lambda_t^{-2/q}r^2\cdot \lambda_t^{2/q}
\)
with conjugate exponents \(q/2\) and \(q/(q-2)\) yields
$
  r^2
  \le
  \frac{2}{q\lambda_t}r^q + \frac{q-2}{q}\lambda_t^\beta
$ for $r \geq 0$.
Hence
\[
  \Phi_t(r)
  \le
  \frac{q-2}{q}L\gamma_t^2\lambda_t^\beta.
\]
This yields the desired dichotomy. Either $\lambda_t > C_\varepsilon \delta_t$,
in which case $
F_t
  =
  \frac{4L\gamma_t\delta_t}{qA\lambda_t}
  <
  \frac{4L\gamma_t}{qA C_\varepsilon}
  =
  \frac{\varepsilon}{2(t+2)}
$, which is equivalent to $h_t < \frac{\varepsilon}{2}$
or else \(\lambda_t\le C_\varepsilon \delta_t\), so
\[
  \Phi_t(r_t)
  \le
  \frac{q-2}{q}L\gamma_t^2 C_\varepsilon^\beta \delta_t^\beta
  =
  R_\varepsilon \gamma_t^2\delta_t^\beta.
\]
Substituting this into the previous descent estimate gives $F_{t+1}
  \le
  (1-\gamma_t)F_t + L\gamma_t^2\delta_t^2 + R_\varepsilon \gamma_t^2\delta_t^\beta
$ and multiplying by $t+3$ gives
\begin{equation}\label{eq:ol-refined-recursion}
  h_{t+1}
  \le
  a_t h_t + c_t\,\operatorname{coeff}_t,
\end{equation}
where for convenience we introduce $a_t \defeq \frac{t(t+3)}{(t+2)^2}$, $c_t \defeq \frac{t+3}{(t+2)^2}$, and $\operatorname{coeff}_t \defeq 4L\delta_t^2 + 4R_\varepsilon\delta_t^\beta$.
Independently, \Cref{prop:standard-descent} and $d_t\le D$ give the standard bound
$h_{t+1}
  \le
  a_t h_t + 2LD^2 \frac{t+3}{(t+2)^2}$.
We now prove that the sublevel set $\{h_t\le \varepsilon\}$ is forward
invariant for $t\ge t_0$. Assume therefore that $h_t\le \varepsilon$.
If the first branch $h_t<\varepsilon/2$ holds, then \eqref{eq:coeff-small}
and the standard recurrence give $h_{t+1} < \varepsilon$ because $a_t\le 1$.
If instead the recurrence \eqref{eq:ol-refined-recursion} holds, then
\eqref{eq:coeff-small} implies
\[
  h_{t+1}
  \le
  a_t \varepsilon + c_t \frac{\varepsilon}{2}
  \le
  \varepsilon,
\]
since $a_t + \frac{c_t}{2} \le 1$.

Next suppose that $t\ge t_0$ and $h_t>\varepsilon$.
Then the first branch is impossible, so \eqref{eq:ol-refined-recursion} must
hold.
Using again \eqref{eq:coeff-small},
$h_{t+1} \le a_t h_t + c_t \frac{\varepsilon}{2}$ follows. Because $1-a_t = \frac{t+4}{(t+2)^2}$, we obtain
\[
  h_t-h_{t+1}
  \ge
  \frac{(t+4)h_t}{(t+2)^2}
  -
  \frac{\varepsilon}{2}\frac{t+3}{(t+2)^2}
  \ge
  \frac{\varepsilon(t+5)}{2(t+2)^2}
  \ge
  \frac{\varepsilon}{2(t+2)}.
\]
Therefore, on every block $N\le t\le 2N$ with $N\ge t_0$,
\begin{equation}\label{eq:block-drop}
  h_{t+1}
  \le
  h_t - \frac{\varepsilon}{4(N+1)}.
\end{equation}

Choose $m\in\NN$ so large that $2LD^2 < m\varepsilon/4$, and define $N_k \defeq 2^k(t_0+1)-1$. If some $n\ge N_m$ satisfied $h_n>\varepsilon$, then the forward invariance would force
$h_t>\varepsilon$ for every $t\in\{t_0,t_0+1,\dots,n\}$.
Applying \eqref{eq:block-drop} on each dyadic block
$N_k\le t\le 2N_k=N_{k+1}-1$ gives
$h_{N_{k+1}} \le h_{N_k} - \frac{\varepsilon}{4}$
  for all $k=0,\dots,m-1$. By induction, $
  h_{N_k} \le 2LD^2 - k\frac{\varepsilon}{4}$ for all $k \leq m$.
At $k=m$ this is negative, contradicting $h_t\ge 0$.
Hence $h_t\le \varepsilon$ for all sufficiently large $t$.
Since $\varepsilon>0$ was arbitrary, $h_t\to 0$ and equivalently $(t+2) F_t \to 0$.
\end{deferredproof}

Following the proof of \cref{thm:local-gap-zgm} we immediately obtain:

\begin{theorem}[LDS + open-loop\leanverified]
\label{thm:ol-set}
Assume that $\cC$ is compact and convex.
Let $(x_t)$ be the Frank--Wolfe iterates generated by the classical open-loop
rule $\gamma_t=2/(t+2)$.
Suppose that $\cC$ together with its LMO satisfies local dual sharpness
around the minimizer set \(\cM\) with some constants \((A,q)\), where \(q\ge2\).
Then
\[
  f(x_t)-f^\star = o(1/t).
\]
\begin{proof}
By \Cref{prop:standard-rate}, $F_t\to 0$.
Hence \Cref{lem:minset-distance} gives
\[
  \dist(x_t,\cM)\to 0.
\]
Choose $p_t\in\cM$ with
\[
  \norm{x_t-p_t} = \dist(x_t,\cM).
\]
Then $\norm{x_t-p_t}\to 0$.
Now apply \Cref{thm:conditional-ol}.
\end{proof}
\end{theorem}

\section{Quantitative Rates under a Local Hölder Error Bound}
\label{sec:quantitative}

The LDS theorems in \Cref{sec:ss-and-ls,sec:open-loop} use only that the iterates
approach the minimizer set, which is essentially optimal without further
structure on \(f\). With a local H\"older error bound, the same mechanisms become
quantitative and yield explicit tail rates after an unoptimized burn-in \(t_0\),
generalizing the uniformly-convex-set H\"older-error-bound regime; see
\Cref{tab:known-fw-rates}.

\begin{definition}[Local Hölder error bound of order \(r\)]
\label{def:local-heb}
Let \(r\ge 2\).
We say that \(f\) satisfies a \emph{local Hölder error bound of order \(r\)}
around the minimizer set \(\cM\) if there exist constants \(B,\rho>0\) such
that
\[
  \dist(x,\cM)^r
  \le
  B\bigl(f(x)-f^\star\bigr)
  \qquad
  \forall x\in\cC
  \text{ with }
  \dist(x,\cM)<\rho.
\]
\end{definition}
Sometimes the above is also written as
$\dist(x,\cM)\le c\bigl(f(x)-f^\star\bigr)^\theta$, which we recover with parameters $r=1/\theta$ and $B=c^r$. Thus quadratic growth, and in particular
$\mu$-strong convexity, gives $r=2$ (one may take $B=2/\mu$), while local
order-$p$ growth gives $r=p$. The restriction $r\ge2$ corresponds to
$0<\theta\le1/2$.

\begin{deferredtheorem}[Quantitative rate under LDS + local HEB\leanverified]
\label{thm:local-gap-heb-rate}
Suppose that $\cC$ together with its LMO satisfies local dual sharpness
around the minimizer set \(\cM\) with constants \((A,q)\), where
\(q\ge 2\).
Assume also that \(f\) satisfies a local Hölder error bound of order \(r\ge 2\)
around \(\cM\) with constants \((B,\rho)\).
Then for the global short-step rule, exact line search, and the classical
open-loop rule, there exist \(t_0\in\NN\) and \(K>0\) such that
$f(x_t)-f^\star
  \le
  \frac{K}{(t-t_0+1)^{rq/(rq-2)}}$ for all $t \geq t_0$.
\end{deferredtheorem}

\begin{deferredproof}[proof:thm:local-gap-heb-rate]
\emph{We first consider short steps or exact line search.}
By \Cref{prop:standard-rate}, \(F_t\to 0\), and then
\Cref{lem:minset-distance} gives $\dist(x_t,\cM)\to 0$.
Choose \(p_t\in\cM\) with $\delta_t \defeq \norm{x_t-p_t} = \dist(x_t,\cM)$.
After increasing \(t_0\) if needed, we may assume that for all \(t\ge t_0\),
$\delta_t < \rho$, $\delta_t^r \le B F_t$, and $F_t \le 1$.
On this tail local dual sharpness and convexity at \(p_t\) give
\[
  A\frac{F_t}{\delta_t} d_t^q \le g_t.
\]

For the short-step and exact-line-search part, if $F_t=0$ at some iteration,
the method has reached an optimum and the claimed rate is immediate. We may
therefore assume $F_t>0$ for all $t$ and define $H_t\defeq1/F_t$ as in the
proof of \Cref{thm:conditional-local-gap}. Then
\Cref{prop:standard-descent} gives
\[
  F_{t+1}
  \le
  F_t - \frac12 \min\left\{g_t,\frac{g_t^2}{L d_t^2}\right\}.
\]
Since \(F_{t+1}\le F_t\), this implies
\[
  H_{t+1}-H_t
  =
  \frac{F_t-F_{t+1}}{F_tF_{t+1}}
  \ge
  \frac{F_t-F_{t+1}}{F_t^2}.
\]
We now split into the same two cases as in \Cref{thm:conditional-local-gap}.
If \(g_t\ge L d_t^2\), then
$F_t-F_{t+1}\ge \frac12 g_t \ge \frac12 F_t$,
so
\[
  H_{t+1}-H_t \ge \frac{1}{2F_t}\ge \frac12 F_t^{-2/(rq)},
  \]
  because \(F_t\le 1\) and \(2/(rq)\le 1\).

If instead \(g_t\le L d_t^2\), then the small-gap branch gives
$
  F_t-F_{t+1}
  \ge
  \frac{g_t^2}{2L d_t^2}$,
and therefore
\[
  H_{t+1}-H_t
  \ge
  \frac{g_t^2}{2L d_t^2 F_t^2}.
\]
Using first \(g_t\ge F_t\) and then \(A(F_t/\delta_t)d_t^q \le g_t\), we obtain
$
  H_{t+1}-H_t
  \geq
  \frac{1}{2L d_t^2}$, and
$H_{t+1}-H_t
  \geq
  \frac{A^2 d_t^{2q-2}}{2L\delta_t^2}.
$
Taking the weighted geometric mean of these two bounds with weights
\((q-1)/q\) and \(1/q\), exactly as in Case 2 of
\Cref{thm:conditional-local-gap}, yields
$
  H_{t+1}-H_t
  \ge
  \frac{A^{2/q}}{2L}\,\delta_t^{-2/q}.
$
We now combine this with the local Hölder error bound \(\delta_t^r\le B F_t\), to obtain
$
  H_{t+1}-H_t
  \ge
  \frac{A^{2/q}}{2L\,B^{2/(rq)}}\,F_t^{-2/(rq)}.
$
Therefore, for all \(t\ge t_0\),
\[
  H_{t+1}-H_t \ge \eta F_t^{-2/(rq)},
  \qquad
  \eta \defeq \min\left\{\frac12,\frac{A^{2/q}}{2L\,B^{2/(rq)}}\right\}.
\]
Equivalently, $
  \frac{1}{F_{t+1}}
  \ge
  \frac{1+\eta F_t^{1-2/(rq)}}{F_t}$,
hence
\[
  F_{t+1}
  \le
  \frac{F_t}{1+\eta F_t^{1-2/(rq)}}.
\]
After increasing \(t_0\) once more, we may assume
\(\eta F_t^{1-2/(rq)}\le 1\) on the tail.
Using \(1/(1+u)\le 1-u/2\) for \(u\in[0,1]\), we obtain
\[
  F_{t+1}
  \le
  F_t - \frac{\eta}{2} F_t^{2-2/(rq)}.
\]
Applying \Cref{lem:power-descent} to the shifted tail yields
\[
  F_t = O\!\left(t^{-rq/(rq-2)}\right),
\]
which is the claimed estimate.

\emph{For the classical open-loop rule, first assume \(q>2\).}
Set \(h_t \defeq (t+2)F_t\),
\(\sigma \defeq \frac{2}{rq-2}\), \(\beta \defeq \frac{2}{q-2}\), and
\(\theta \defeq \frac{\beta}{r} = \frac{2}{r(q-2)}\).
By \Cref{prop:standard-rate}, \(0 \le h_t \le 2LD^2\) for all \(t \ge 1 \), and
we work after a burn-in \(t_0 \ge 1\).
The rescaling by \(t+2\) factors out the classical \(O(1/t)\) rate. Thus the
desired estimate \(F_t=O(t^{-rq/(rq-2)})\) is equivalent to proving
\[
  h_t = O(t^{-\sigma}).
\]
We prove this by a moving-barrier induction with
\(\varepsilon_t\defeq K(t+2)^{-\sigma}\). The role of \(K\) is only to make the
barrier high enough after a finite burn-in; it is not optimized.

As above, choose \(p_t\in\cM\) with
\(\delta_t \defeq \norm{x_t-p_t}=\dist(x_t,\cM)\).
Since \(F_t\to 0\) and \(\dist(x_t,\cM)\to 0\), after shifting the index we may
assume
\[
  \delta_t<\rho,
  \qquad
  \delta_t^r \le B F_t = B\frac{h_t}{t+2}
  \qquad \forall t\ge 0.
\]

For every \(t\), repeat the one-step argument from the proof of
\Cref{thm:conditional-ol}, but with the level \(\varepsilon\) replaced by
\(\varepsilon_t \defeq \frac{K}{(t+2)^\sigma}\). This yields the same
dichotomy: either \(h_t \le \varepsilon_t/2\), or
\[
  h_{t+1}
  \le
  a_t h_t + c_t\bigl(4L\delta_t^2 + 4R_t\delta_t^\beta\bigr),
\]
where \(a_t \defeq \frac{t(t+3)}{(t+2)^2}\),
\(c_t \defeq \frac{t+3}{(t+2)^2}\), \(R_t \defeq C_0 \varepsilon_t^{-\beta}\),
and \(C_0>0\) depends only on \((A,L,q)\).
The first alternative is the easy branch: the scaled gap is already well below
the barrier. The second alternative is the branch where the LDS correction kicks in; the price is the two error terms involving the distance
\(\delta_t\) to the minimizer set.

We prove by induction that \(h_t\le \varepsilon_t\) for all large \(t\).
Assume \(h_t\le \varepsilon_t\).
Then the local Hölder error bound gives
\(\delta_t^r \le B\frac{\varepsilon_t}{t+2}
= \frac{BK}{(t+2)^{1+\sigma}}\), hence
\[
  \delta_t^2 \le B^{2/r}K^{2/r}(t+2)^{-2(1+\sigma)/r},
  \qquad
  \delta_t^\beta \le B^\theta K^\theta (t+2)^{-(1+\sigma)\theta}.
\]
Because \(r\ge 2\), once \(K\ge 1\) we also have \(K^{2/r}\le K\) and
\(K^{-\beta}K^\theta = K^{-(r-1)\theta}\le 1\).
Consequently there exist constants \(C_1,C_2>0\), depending only on
\((A,B,L,q,r)\), such that on the refined branch
\[
  h_{t+1}
  \le
  a_t h_t
  +
  \frac{C_1 K}{(t+2)^{1+\sigma+\tau}}
  +
  \frac{C_2}{(t+2)^{1+\sigma}},
\]
where \(\tau \defeq \frac{2(q-1)}{rq-2} > 0\), as we will argue now: first observe that for the distance error term $c_t\,4L\delta_t^2$, we have
\[
  c_t\,4L\delta_t^2
  \le
  \frac{2}{t+2}\,4L B^{2/r} K^{2/r}(t+2)^{-2(1+\sigma)/r}
  \le
  \frac{C_1 K}{(t+2)^{1+\sigma+\tau}},
\]
since \(2(1+\sigma)/r-\sigma=\tau\), while for the LDS error term $c_t\,4R_t\delta_t^\beta$ we have
\[
  c_t\,4R_t\delta_t^\beta
  \le
  \frac{2}{t+2}\,4C_0 K^{-\beta}(t+2)^{\sigma\beta}
  B^\theta K^\theta (t+2)^{-(1+\sigma)\theta},
\]
and \(\theta\bigl(1-(r-1)\sigma\bigr)=\sigma\) gives precisely the exponent
\(1+\sigma\) in the denominator.

Next we will estimate the slack left by the homogeneous part
\(h_{t+1}\le a_t h_t\) of the scaled open-loop recurrence, in which we will have to fit our two error terms. To this end observe \(a_t \le 1-\frac{1}{t+2}\), while Bernoulli's inequality gives
\[
  \varepsilon_{t+1}
  =
  \frac{K}{(t+3)^\sigma}
  =
  \frac{K}{(t+2)^\sigma}\left(1+\frac{1}{t+2}\right)^{-\sigma}
  \ge
  \varepsilon_t\left(1-\frac{\sigma}{t+2}\right),
\]
hence
\[
  \varepsilon_{t+1} - a_t\varepsilon_t
  \ge
  \frac{(1-\sigma)K}{(t+2)^{1+\sigma}}.
\]
Thus, on the refined branch, once \(K\) and \(t_0\) are chosen large enough, the two error terms can be absorbed as we will argue now: Choose \(K\ge 1\) so large that \(C_2 \le \frac{1-\sigma}{4}K\). Then choose
\(t_0\) so large that for all \(t\ge t_0\), we have
\[
  \frac{C_1}{(t+2)^\tau} \le \frac{1-\sigma}{4}
  \qquad \text{and} \qquad
  2LD^2\,\frac{t+3}{(t+2)^2}
  \le
  \left(1-2^{\sigma-1}\right)\varepsilon_{t+1}.
\]
Since \(\sigma<1\), the latter is possible because the left-hand side is
\(O(t^{-1})\) while \(\varepsilon_{t+1}=K(t+3)^{-\sigma}\).
Enlarging \(K\) once more if necessary, we may assume
\(h_t\le \varepsilon_t\) for all \(t\in\{t_0,t_0+1,\dots,2t_0\}\).

Now fix \(t\ge t_0\) and assume \(h_t\le \varepsilon_t\).
If the small branch holds, then \(h_t\) starts with a factor-\(1/2\) margin
below the barrier, and the standard recurrence gives
\[
  h_{t+1}
  \le
  a_t h_t + 2LD^2\,\frac{t+3}{(t+2)^2}
  \le
  \frac{\varepsilon_t}{2}
  +
  \left(1-2^{\sigma-1}\right)\varepsilon_{t+1}
  \le
  \varepsilon_{t+1},
\]
because \(\varepsilon_t/2 \le 2^{\sigma-1}\varepsilon_{t+1}\).
If the refined branch holds, then the barrier slack absorbs the distance and LDS error
terms, and the previous two display formulas imply
\[
  h_{t+1}
  \le
  a_t \varepsilon_t
  +
  \frac{1-\sigma}{2}\frac{K}{(t+2)^{1+\sigma}}
  \le
  \varepsilon_{t+1}.
\]
Thus \(h_t\le \varepsilon_t\) propagates forward for all \(t\ge t_0\), and
therefore \(h_t = O\!\left(t^{-2/(rq-2)}\right)\). Since
\(F_t=h_t/(t+2)\), we obtain \(F_t = O\!\left(t^{-rq/(rq-2)}\right)\), as
claimed. \medskip

It remains to handle the quadratic open-loop branch $q=2$. As before, set
\[
  a_t\defeq\frac{t(t+3)}{(t+2)^2},
  \qquad
  c_t\defeq\frac{t+3}{(t+2)^2},
  \qquad
  C_{\mathrm s}\defeq\frac{4L}{A},
  \qquad
  B_0\defeq\max\{F_0,LD^2\}.
\]
The standard rate gives $0\le h_t\le 2B_0$ for every $t$. Repeating the quadratic one-step
calculation from \Cref{thm:conditional-ol} gives, on the local-HEB tail, two
branches; we refer to the first as the small branch and the second as the refined branch:
\begin{equation}\label{eq:q2-heb-dichotomy}
  h_t\le C_{\mathrm s}\delta_t
  \qquad\text{or}\qquad
  h_{t+1}\le a_t h_t+c_t\,4L\delta_t^2.
\end{equation}
These two branches hold together with the standard recurrence
\begin{equation}\label{eq:q2-heb-standard}
  h_{t+1}\le a_t h_t+2LD^2c_t
\end{equation}
and the local error bound
$\delta_t^r\le B h_t/(t+2)$ from our assumption.

\emph{Subcase $r>2$.}
Let
\[
  \sigma\defeq\frac{1}{r-1}\in(0,1),
  \qquad
  S_0\defeq(C_{\mathrm s}^rB)^\sigma,
  \qquad
  K_0\defeq\max\{1,2S_0\}.
\]
In a first step we choose $T$ beyond the local-HEB burn-in so that for every $t\ge T$,
\begin{equation}\label{eq:q2-rgt2-tail-choices}
  \frac{8LB^{2/r}}{(t+2)^\sigma}\le1-\sigma,
  \qquad
  \frac{8LD^2}{(t+3)^{1-\sigma}}
  \leq \bigl(1-2^{\sigma-1}\bigr)K_0.
\end{equation}
Finally set
\[
  K\defeq\max\{K_0,2B_0(T+2)^\sigma\},
  \qquad
  \varepsilon_t\defeq\frac{K}{(t+2)^\sigma}.
\]
Then $h_T\le\varepsilon_T$ by definition of $K$:
$h_T\leq 2B_0 \leq 2B_0(T+2)^\sigma/(T+2)^\sigma \leq
K/(T+2)^\sigma = \varepsilon_T$. This is our induction starting point to prove
$h_t \leq \varepsilon_t$ by induction.

In the small branch of \eqref{eq:q2-heb-dichotomy}, combining
$h_t \leq C_{\mathrm s} \delta_t$ with HEB gives for $h_t>0$
\[
  h_t^{r-1}\le\frac{C_{\mathrm s}^rB}{t+2},
  \qquad\text{hence}\qquad
  h_t\le\frac{S_0}{(t+2)^\sigma}\le\frac{\varepsilon_t}{2},
\]
where the last inequality uses $K \geq K_0 \geq 2S_0$; for $h_t=0$ the same bound
is immediate. Using this bound together with $a_t \leq 1$ and the standard recursion from
\eqref{eq:q2-heb-standard}, we get
\[
  h_{t+1} \leq \frac{\varepsilon_t}{2} + 2LD^2c_t,
\]
and we need to bound the right-hand side by $\varepsilon_{t+1}$. For this we
rewrite and bound the curvature term
\[
  2LD^2c_t = \frac{2LD^2}{(t+3)^{1-\sigma}}
   \left(\frac{t+3}{t+2}\right)^2
   \frac{1}{(t+3)^\sigma}
  \leq \frac{8LD^2}{(t+3)^{1-\sigma}}
   \frac{1}{(t+3)^\sigma},
\]
as $(t+3) / (t+2) \leq 2$. Combining this with the second condition in
\eqref{eq:q2-rgt2-tail-choices} and $K \geq K_0$, we obtain
\[
  2LD^2c_t
  \le\bigl(1-2^{\sigma-1}\bigr)\frac{K}{(t+3)^\sigma}
  =\bigl(1-2^{\sigma-1}\bigr)\varepsilon_{t+1}.
\]
Moreover,
\[
  \frac{\varepsilon_t}{2}
  =\frac{K}{2(t+2)^\sigma}
  =\frac{1}{2} \left( \frac{t+3}{t+2} \right)^\sigma \varepsilon_{t+1}
  \le2^{\sigma - 1} \varepsilon_{t+1}.
\]
Combining the two bounds gives
\[
  h_{t+1} \leq \frac{\varepsilon_t}{2} + 2LD^2c_t
  \leq \varepsilon_{t+1}.
\]

In the refined branch, the induction hypothesis, the definition of
$\varepsilon_t$, and the HEB imply
\[
  \delta_t^2
  \le\left(\frac{B\varepsilon_t}{t+2}\right)^{2/r}
  =\left(\frac{BK}{(t+2)^{1+\sigma}}\right)^{2/r}
  =\frac{(BK)^{2/r}}{(t+2)^{2(1+\sigma)/r}}
  =\frac{B^{2/r}K^{2/r}}{(t+2)^{2\sigma}},
\]
where the last equality uses $\sigma=1/(r-1)$, so that
$2(1+\sigma)/r=2\sigma$. Since $K\ge1$, one has $K^{2/r}\le K$, and since
$c_t\le2/(t+2)$, the first condition in
\eqref{eq:q2-rgt2-tail-choices} gives
\[
  4Lc_t\delta_t^2
  \le\frac{(1-\sigma)K}{(t+2)^{1+\sigma}}.
\]
On the other hand,
\[
  \varepsilon_{t+1}
  =\varepsilon_t\left(1+\frac{1}{t+2}\right)^{-\sigma}
  \ge\varepsilon_t\left(1-\frac{\sigma}{t+2}\right)
\]
by Bernoulli's inequality. Using
$a_t\le1-1/(t+2)$ therefore yields
\[
  \varepsilon_{t+1}-a_t\varepsilon_t
  \ge\frac{(1-\sigma)K}{(t+2)^{1+\sigma}}.
\]
Thus the refined recurrence and the induction hypothesis give
\begin{align*}
  h_{t+1}
  &\le a_t h_t+4Lc_t\delta_t^2\\
  &\le a_t\varepsilon_t
    +\frac{(1-\sigma)K}{(t+2)^{1+\sigma}}\\
  &\le a_t\varepsilon_t
    +\bigl(\varepsilon_{t+1}-a_t\varepsilon_t\bigr)
   =\varepsilon_{t+1}.
\end{align*}
Together with the base case and the small-branch estimate, this proves
$h_t\le\varepsilon_t$ for every $t\ge T$. Consequently,
$h_t=O(t^{-1/(r-1)})$ and $F_t=O(t^{-r/(r-1)})$.

\emph{Subcase $r=2$.}
At $r=2$, the exponent $\sigma=1$ makes the first-order slack used above
vanish, so we need to adjust our argument using a shifted barrier. Define
\[
  C_q\defeq4LB,
  \qquad
  s\defeq C_q^2+C_q+10,
\]
and choose $T$ beyond the local-HEB burn-in and large enough that
\[
  T\ge\lceil4C_q+4\rceil,
\]
and further set
\[
  K_0\defeq
  \max\{1,4C_{\mathrm s}^2Bs,8LD^2(s+1)\},
  \qquad
  K\defeq\max\{K_0,2B_0(T+s)\},
  \qquad
  \varepsilon_t\defeq\frac{K}{t+s}.
\]
Similar to before, the definition of $K$ initializes the induction:
\[
  h_T\le2B_0\le\frac{K}{T+s}=\varepsilon_T.
\]
Assume now that $t\ge T$ and $h_t\le\varepsilon_t$. If we are in the small branch and $h_t>0$, then the local HEB gives
\[
  h_t^2
  \le C_{\mathrm s}^2\delta_t^2
  \le\frac{C_{\mathrm s}^2B}{t+2}h_t,
\]
so
\[
  h_t\le\frac{C_{\mathrm s}^2B}{t+2}
  \le\frac{C_{\mathrm s}^2Bs}{t+s}
  \le\frac{K}{4(t+s)}
  =\frac{\varepsilon_t}{4},
\]
where we used $s/(t+s)\ge1/(t+2)$ because $s\ge1$ in the second inequality and
$K\ge4C_{\mathrm s}^2Bs$ in the third inequality. The case $h_t=0$ is again immediate. Since
$a_t\le1$, the standard recurrence yields
\[
  h_{t+1}\le\frac{\varepsilon_t}{4}+2LD^2c_t.
\]
The first term fits into half of the next barrier $\varepsilon_{t+1}$ because $t+s\ge1$:
\[
  \frac{\varepsilon_t}{4}
  =\frac{K}{4(t+s)}
  \le\frac{K}{2(t+s+1)}
  =\frac{\varepsilon_{t+1}}{2}.
\]
For the standard curvature term, we use $c_t\le2/(t+2)$ and
$K\ge8LD^2(s+1)$ to obtain
\begin{align*}
  2LD^2c_t
  &\le\frac{4LD^2}{t+2}\\
  &\le\frac{4LD^2(s+1)}{t+s+1}\\
  &\le\frac{K}{2(t+s+1)}
   =\frac{\varepsilon_{t+1}}{2}.
\end{align*}
Therefore the standard recurrence gives $h_{t+1}\le\varepsilon_{t+1}$ in the
small branch.
If the refined branch holds, then the local HEB gives
$\delta_t^2\le Bh_t/(t+2)$, and consequently
\begin{align*}
  h_{t+1}
  &\le a_t h_t+4Lc_t\delta_t^2\\
  &\le\left(
    \frac{t(t+3)}{(t+2)^2}
    +C_q\frac{t+3}{(t+2)^3}
  \right)h_t.
\end{align*}
The choice $s=C_q^2+C_q+10$ ensures the elementary coefficient bound
\[
  \frac{\frac{t(t+3)}{(t+2)^2}
    +C_q\frac{t+3}{(t+2)^3}}{t+s}
  \le\frac{1}{t+s+1},
\]
whenever $t\ge4C_q+4$, which we prove below via a somewhat technical calculation where in essence $s$ as shift is chosen so that the above inequality holds.
Since $T\ge4C_q+4$, the induction hypothesis therefore gives
\begin{align*}
  h_{t+1}
  &\le K\,
  \frac{\frac{t(t+3)}{(t+2)^2}
    +C_q\frac{t+3}{(t+2)^3}}{t+s}\\
  &\le\frac{K}{t+s+1}
   =\varepsilon_{t+1}.
\end{align*}
The base case and both branches prove $h_t\le K/(t+s)=O(1/t)$ for every
$t\ge T$. Hence
\[
  F_t=\frac{h_t}{t+2}=O(t^{-2}).
\]
Combining the two subcases gives
\[
  F_t=O\!\left(t^{-r/(r-1)}\right),
\]
which is exactly $O(t^{-rq/(rq-2)})$ at $q=2$.

It remains to prove
\[
  \frac{\frac{t(t+3)}{(t+2)^2}
    +C_q\frac{t+3}{(t+2)^3}}{t+s}
  \le\frac{1}{t+s+1},
\]
whenever $t\ge4C_q+4$. To see this, set
\[
  \kappa_t\defeq
  \frac{t(t+3)}{(t+2)^2}
  +C_q\frac{t+3}{(t+2)^3}.
\]
Since $t+s>0$, the desired inequality is equivalent to
\[
  \kappa_t
  \le\frac{t+s}{t+s+1}
  =1-\frac{1}{t+s+1}.
\]
We first lower-bound the gap to $1$. We have
\begin{align*}
  1 - \kappa_t
  &= \frac{t+4}{(t+2)^2}
    - C_q\frac{t+3}{(t+2)^3}\\
  &\geq \frac{1}{t+2} - \frac{2C_q}{(t+2)^2} = \frac{t+2-2C_q}{(t+2)^2},
\end{align*}
where we used $(t+3)/(t+2)\le2$. It therefore remains to show
\[
  \frac{t+2-2C_q}{(t+2)^2}\ge\frac{1}{t+s+1}.
\]
After cross-multiplication, this is equivalent to
\[
  \bigl((t+2)-2C_q\bigr)(t+s+1)-(t+2)^2\ge0.
\]
Using $t+s+1=(t+2)+(s-1)$ and keeping $(t+2)$ grouped, the left-hand side is
\begin{align*}
  &\bigl((t+2)-2C_q\bigr)\bigl((t+2)+(s-1)\bigr)-(t+2)^2\\
  &\qquad=(t+2)(s-1-2C_q)-2C_q(s-1)\\
  &\qquad=(t+2)(C_q^2-C_q+9)-2C_q(C_q^2+C_q+9),
\end{align*}
where the last equality uses $s=C_q^2+C_q+10$. As $t\ge4C_q+4$, one has
$t+2\ge4C_q+6$ and moreover $C_q^2-C_q+9>0$, so that
\begin{align*}
  &(t+2)(C_q^2-C_q+9)-2C_q(C_q^2+C_q+9)\\
  &\qquad\ge(4C_q+6)(C_q^2-C_q+9)-2C_q(C_q^2+C_q+9)\\
  &\qquad=2C_q^3+12C_q+54>0.
\end{align*}
Consequently,
\[
  1-\kappa_t\ge\frac{1}{t+s+1},
\]
which gives $\kappa_t\le(t+s)/(t+s+1)$ and therefore proves the claimed
coefficient bound.
\end{deferredproof}

\section{A Computational Example}
\label{subsec:computational-results}

We consider the feasible region from \Cref{rem:local-gap-examples}, which satisfies LDS but is not uniformly convex:
\[
  \cC_{\mathrm{stad}}
  \defeq
  \bigl([-1,1]\times\{0\}\bigr)+B_2(0,1),
  \qquad
  p\defeq (2,0),
\]
and the objective family, convex and \(1\)-smooth on \(\cC_{\mathrm{stad}}\),
\[
  f_c(x_1,x_2)
  \defeq
  \frac12 x_2^2 + \psi(c-x_1),
  \qquad
  \psi(u)\defeq u-\arctan(u),
  \qquad
  c\ge 2.
\]
On \(\cC_{\mathrm{stad}}\) one has \(c-x_1\ge0\), and the same restriction can be obtained from a global convex smooth extension if desired. For \(c=2\), \(p\) is the minimizer and \(\nabla f_2(p)=0\), which is the interesting case. In \Cref{fig:stadium-local-gap-regimes} we depict the resulting primal-gap trajectories for short steps, exact line search, and for the classical open-loop rule from the starting point \(x_0=e_2\), using \(L=1\) in the short-step rule, together with the stadium geometry and the level sets of \(f_2\). The minimizer lies on the rounded cap where we show that LDS holds in the appendix; the flat top and bottom segments show why the stadium as a whole is not uniformly convex and why LDS does not hold for an arbitrary minimizer location. Analogous arbitrarily high-dimensional examples are obtained from capsules or truncated Euclidean balls; we use \(\RR^2\) only for visualization. A second example with known exponents (a quadratic objective on an $\ell_4$ ball with LDS power $q=4$ and HEB order $r=2$) is reported in \Cref{sec:l4-additional-example}; the predicted rate is $O(t^{-4/3})$.

\begin{figure}[t]
  \centering
  \includegraphics[width=\textwidth]{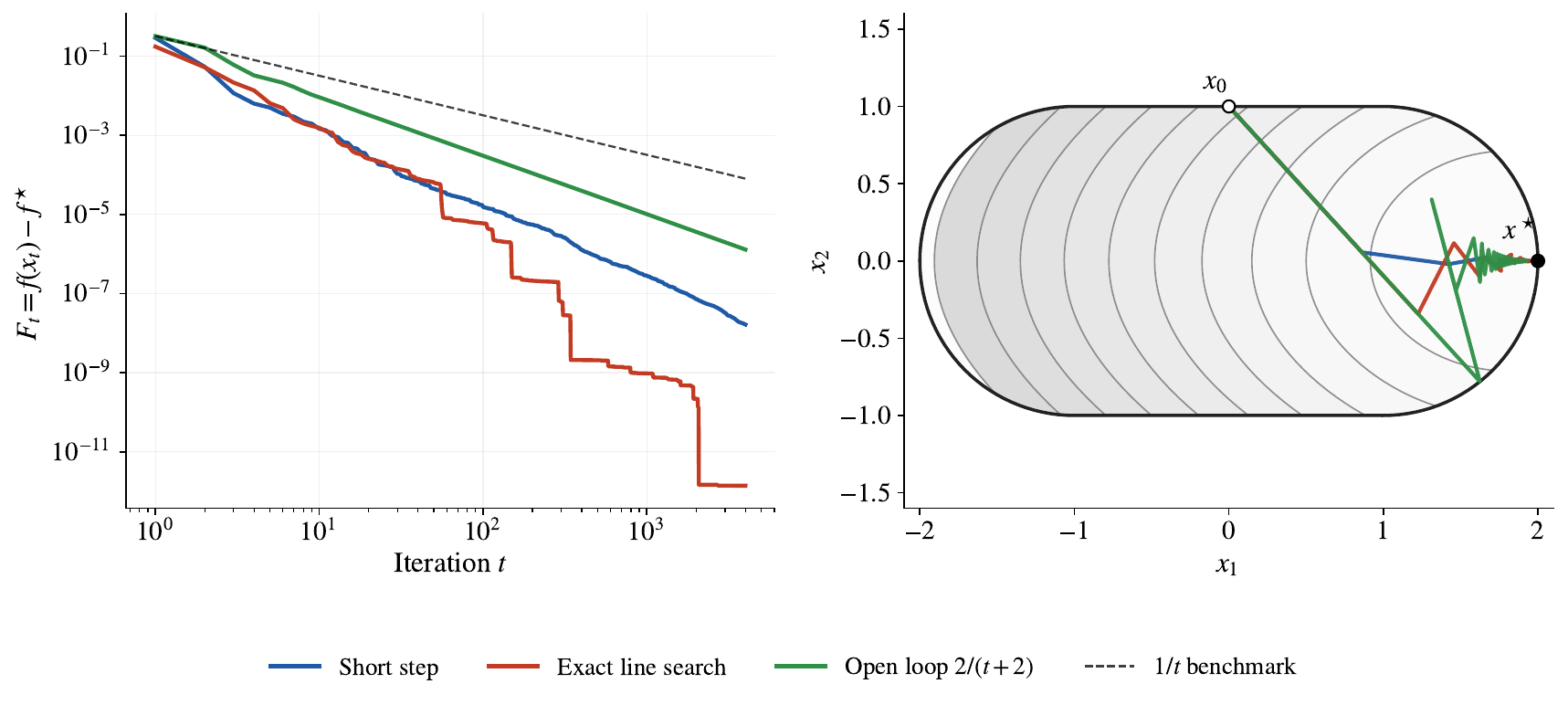}
  \caption{Primal-gap decay and geometry for the stadium.
  The left panel shows primal gaps for iterations \(t\ge 1\) on a log-log plot.
  The right panel shows the feasible region, the level sets of \(f_2\)
  restricted to \(\cC_{\mathrm{stad}}\), and the corresponding Frank--Wolfe
  trajectories. The finite-horizon curves fall below the displayed $1/t$ guide; the plot is illustrative and does not by itself establish the asymptotic little-$o$ statement.}
  \label{fig:stadium-local-gap-regimes}
\end{figure}

\paragraph{Limitations.} LDS is a condition on a reference region and a fixed
LMO selector. In our theorems that region is the objective-dependent minimizer
set, so \emph{verifying} LDS may require knowing where the minimizers lie.
However, the Frank--Wolfe algorithm automatically adapts when LDS holds,
without requiring such verification. Our results are of the form
\(tF_t\to0\) for a given LDS instance, not a uniform rate for
\(\sup_I F_t(I)\) over a class; the associated burn-in is instance-dependent
and can be arbitrarily long. Our explicit polynomial rates require an
additional objective condition such as a local Hölder error bound and remain
tail estimates after an unoptimized burn-in. Nevertheless, a finite-order HEB
is not necessary for a quantified nonpolynomial improvement over \(1/t\):
\Cref{prop:superflat-short-step} provides an LDS instance without any
finite-order HEB whose short-step tail is
\(F_t\sim L/[t(\log t)^4]\). Flat exposed faces obstruct LDS in standard
boundary and low-rank regimes for spectrahedra and nuclear-norm balls.

\IfNotNeuripsBuild{%
\section*{Acknowledgements}
This research was conducted in the context of the Agentic AI for Mathematics (EF-LiOpt-3) project at the Berlin Mathematics Research Center MATH$^+$ (EXC-2046/2, project ID 390685689), funded by the Deutsche Forschungsgemeinschaft (DFG, German Research Foundation) under Germany's Excellence Strategy. Significant parts of the auxiliary Lean 4 formalization were produced using an auto-formalization extension of our Agentic Researcher framework \citep{zimmer2026agentic}. I thank Christoph Spiegel for independently checking the Lean verification code and for improving both the structure of the Lean code and the guidelines for the auto-formalization. All remaining issues are solely mine.
 }

\bibliography{refs}

\begin{thebibliography}{25}
\providecommand{\natexlab}[1]{#1}
\providecommand{\url}[1]{\texttt{#1}}
\expandafter\ifx\csname urlstyle\endcsname\relax
  \providecommand{\doi}[1]{doi: #1}\else
  \providecommand{\doi}{doi: \begingroup \urlstyle{rm}\Url}\fi

\bibitem[Braun et~al.(2019)Braun, Pokutta, Tu, and Wright]{braun2019blended}
G{\'a}bor Braun, Sebastian Pokutta, Dan Tu, and Stephen Wright.
\newblock Blended conditional gradients.
\newblock In \emph{International conference on machine learning}, pages 735--743. PMLR, 2019.

\bibitem[Braun et~al.(2025)Braun, Carderera, Combettes, Hassani, Karbasi, Mokhtari, and Pokutta]{braun2025cgm}
G{\'a}bor Braun, Alejandro Carderera, Cyrille~W. Combettes, Hamed Hassani, Amin Karbasi, Aryan Mokhtari, and Sebastian Pokutta.
\newblock \emph{Conditional Gradient Methods: From Core Principles to {AI} Applications}.
\newblock {MOS-SIAM} Series on Optimization. {Society for Industrial and Applied Mathematics}, 2025.
\newblock ISBN 978-1-61197-856-8.
\newblock \doi{10.1137/1.9781611978568}.

\bibitem[Canon and Cullum(1968)]{Canon_FWbound68}
M.~D. Canon and C.~D. Cullum.
\newblock A tight upper bound on the rate of convergence of {F}rank--{W}olfe algorithm.
\newblock \emph{{SIAM} Journal on Control}, 6\penalty0 (4):\penalty0 509--516, 1968.
\newblock \doi{10.1137/0306032}.

\bibitem[Demyanov and Rubinov(1970)]{demyanov1970approximate}
Vladimir~F. Demyanov and Alexander~M. Rubinov.
\newblock \emph{{Approximate Methods in Optimization Problems}}, volume~32 of \emph{Modern Analytic and Computational Methods in Science and Mathematics}.
\newblock American Elsevier Publishing Company, New York, 1970.

\bibitem[Dunn(1979)]{dunn1979rates}
J.~C. Dunn.
\newblock {Rates of Convergence for Conditional Gradient Algorithms near Singular and Nonsingular Extremals}.
\newblock \emph{SIAM Journal on Control and Optimization}, 17\penalty0 (2):\penalty0 187--211, 1979.

\bibitem[Frank and Wolfe(1956)]{frank1956algorithm}
Marguerite Frank and Philip Wolfe.
\newblock An algorithm for quadratic programming.
\newblock \emph{Naval Research Logistics Quarterly}, 3\penalty0 (1--2):\penalty0 95--110, 1956.

\bibitem[Garber and Hazan(2015)]{garber2015faster}
Dan Garber and Elad Hazan.
\newblock Faster rates for the {Frank-Wolfe} method over strongly-convex sets.
\newblock In \emph{International Conference on Machine Learning}, pages 541--549. PMLR, 2015.

\bibitem[Grimmer and Liu(2026)]{grimmer2026uniform}
Benjamin Grimmer and Ning Liu.
\newblock Lower bounds for linear minimization oracle methods optimizing over strongly convex sets.
\newblock \emph{arXiv preprint arXiv:2602.22608}, 2026.

\bibitem[Gu{\'e}lat and Marcotte(1986)]{gm86}
Jacques Gu{\'e}lat and Patrice Marcotte.
\newblock Some comments on {W}olfe's `away step'.
\newblock \emph{Mathematical Programming}, 35\penalty0 (1):\penalty0 110--119, 1986.
\newblock \doi{10.1007/BF01589445}.

\bibitem[Halbey et~al.(2026)Halbey, Deza, Zimmer, Roux, Stellato, and Pokutta]{halbey2026lower}
Jannis Halbey, Daniel Deza, Max Zimmer, Christophe Roux, Bartolomeo Stellato, and Sebastian Pokutta.
\newblock {Lower Bounds for Frank--Wolfe on Strongly Convex Sets}.
\newblock \emph{arXiv preprint arXiv:2602.04378}, 2026.

\bibitem[Jaggi(2013)]{jaggi2013revisiting}
Martin Jaggi.
\newblock Revisiting {F}rank--{W}olfe: Projection-free sparse convex optimization.
\newblock In \emph{Proceedings of the 30th International Conference on Machine Learning}, volume~28 of \emph{{JMLR} Workshop and Conference Proceedings}, pages 427--435. JMLR.org, 2013.
\newblock URL \url{http://proceedings.mlr.press/v28/jaggi13.html}.

\bibitem[Kerdreux et~al.(2021{\natexlab{a}})Kerdreux, d'Aspremont, and Pokutta]{kerdreux2021localglobal}
Thomas Kerdreux, Alexandre d'Aspremont, and Sebastian Pokutta.
\newblock {Local and Global Uniform Convexity Conditions}.
\newblock \emph{arXiv preprint arXiv:2102.05134}, 2021{\natexlab{a}}.

\bibitem[Kerdreux et~al.(2021{\natexlab{b}})Kerdreux, d'Aspremont, and Pokutta]{kerdreux2021projection}
Thomas Kerdreux, Alexandre d'Aspremont, and Sebastian Pokutta.
\newblock {Projection-Free Optimization on Uniformly Convex Sets}.
\newblock In \emph{Proceedings of the 24th International Conference on Artificial Intelligence and Statistics (AISTATS)}, volume 130 of \emph{Proceedings of Machine Learning Research}, pages 19--27, 2021{\natexlab{b}}.
\newblock URL \url{https://proceedings.mlr.press/v130/kerdreux21a.html}.

\bibitem[Kerdreux et~al.(2022)Kerdreux, d'Aspremont, and Pokutta]{kerdreux2018restarting}
Thomas Kerdreux, Alexandre d'Aspremont, and Sebastian Pokutta.
\newblock Restarting {F}rank--{W}olfe: Faster rates under {H{\"o}lderian Error Bounds}.
\newblock \emph{Journal of Optimization Theory and Applications}, 192:\penalty0 799--829, 2022.
\newblock \doi{10.1007/s10957-021-01989-7}.

\bibitem[Lacoste-Julien and Jaggi(2015)]{lacoste2015global}
Simon Lacoste-Julien and Martin Jaggi.
\newblock On the global linear convergence of {Frank-Wolfe} optimization variants.
\newblock \emph{Advances in neural information processing systems}, 28, 2015.

\bibitem[Lan(2013)]{lan2013complexity}
Guanghui Lan.
\newblock The complexity of large-scale convex programming under a linear optimization oracle.
\newblock Technical report, Department of Industrial and Systems Engineering, University of Florida, 2013.
\newblock URL \url{https://optimization-online.org/2013/05/3904/}.

\bibitem[Levitin and Polyak(1966)]{levitin1966constrained}
Evgeny~S. Levitin and Boris~T. Polyak.
\newblock Constrained minimization methods.
\newblock \emph{U.S.S.R. Computational Mathematics and Mathematical Physics}, 6\penalty0 (5):\penalty0 1--50, 1966.
\newblock \doi{10.1016/0041-5553(66)90114-5}.

\bibitem[Moura and Ullrich(2021)]{moura2021lean}
Leonardo~de Moura and Sebastian Ullrich.
\newblock The lean 4 theorem prover and programming language.
\newblock In \emph{International Conference on Automated Deduction}, pages 625--635. Springer, 2021.

\bibitem[Pe{\~n}a(2023)]{pena2023affine}
Javier~F. Pe{\~n}a.
\newblock {Affine Invariant Convergence Rates of the Conditional Gradient Method}.
\newblock \emph{{SIAM} Journal on Optimization}, 33\penalty0 (4):\penalty0 2654--2674, 2023.

\bibitem[Pokutta(2024)]{pokutta2023short}
Sebastian Pokutta.
\newblock {The Frank--Wolfe Algorithm: A Short Introduction}.
\newblock \emph{Jahresbericht der Deutschen Mathematiker-Vereinigung}, 126:\penalty0 3--35, 2024.
\newblock \doi{10.1365/s13291-023-00275-x}.

\bibitem[Tsuji et~al.(2022)Tsuji, Tanaka, and Pokutta]{tsuji2022pairwise}
Kazuma~K Tsuji, Ken'ichiro Tanaka, and Sebastian Pokutta.
\newblock Pairwise conditional gradients without swap steps and sparser kernel herding.
\newblock In \emph{International Conference on Machine Learning}, pages 21864--21883. PMLR, 2022.

\bibitem[Wirth et~al.(2023)Wirth, Kerdreux, and Pokutta]{WKP2022}
Elias Wirth, Thomas Kerdreux, and Sebastian Pokutta.
\newblock {Acceleration of Frank-Wolfe Algorithms with Open-Loop Step-Sizes}.
\newblock In \emph{Proceedings of the 26th International Conference on Artificial Intelligence and Statistics (AISTATS)}, volume 206 of \emph{Proceedings of Machine Learning Research}, pages 77--100. PMLR, 2023.
\newblock URL \url{https://proceedings.mlr.press/v206/wirth23a.html}.

\bibitem[Wirth et~al.(2025)Wirth, Pe{\~n}a, and Pokutta]{WPP2023}
Elias Wirth, Javier Pe{\~n}a, and Sebastian Pokutta.
\newblock {Accelerated Affine-Invariant Convergence Rates of the Frank-Wolfe Algorithm with Open-Loop Step-Sizes}.
\newblock \emph{Mathematical Programming}, 214\penalty0 (1):\penalty0 201--245, 2025.
\newblock \doi{10.1007/s10107-024-02180-2}.

\bibitem[Wolfe(1970)]{wolfe70}
Philip Wolfe.
\newblock Convergence theory in nonlinear programming.
\newblock In \emph{Integer and Nonlinear Programming}, pages 1--36. North-Holland, 1970.

\bibitem[Zimmer et~al.(2026)Zimmer, Pelleriti, Roux, and Pokutta]{zimmer2026agentic}
Max Zimmer, Nico Pelleriti, Christophe Roux, and Sebastian Pokutta.
\newblock {The Agentic Researcher: A Practical Guide to AI-Assisted Research in Mathematics and Machine Learning}.
\newblock \emph{arXiv preprint arXiv:2603.15914}, 2026.

\end{thebibliography}
\bibliographystyle{plainnat}

\clearpage

\appendix

\section{Appendix}

\IfNeuripsBuild{%
\begin{table}[!ht]
\footnotesize
\centering
\setlength{\tabcolsep}{4pt}
\renewcommand{\arraystretch}{1.12}
\begin{tabularx}{\textwidth}{@{}>{\raggedright\arraybackslash}p{1.25cm}>{\raggedright\arraybackslash}p{1.9cm}>{\raggedright\arraybackslash}X>{\raggedright\arraybackslash}p{3cm}>{\raggedright\arraybackslash}p{3.00cm}c@{}}
\toprule
$f$ & $\cC$ & minimizer / regime & rule & bound & key \\
\midrule
\multicolumn{6}{@{}l}{\emph{\textbf{Classical compact-convex regime}}}\\
convex & convex & unrestricted & OL / SS / LS & $O(1/t)$ & [A] \\
convex & convex & unrestricted & FW / FO-LMO & $\Omega(1/t)$ & [B] \\
\addlinespace[0.2em]
\multicolumn{6}{@{}l}{\emph{\textbf{Location-based acceleration}}}\\
SC & convex & $x^\star \in \interior(\cC)$ & SS / LS & $O(e^{-rt})$ & [C] \\
convex & UC$(q)$ & LBG, $\inf_{x\in\cC}\norm{\nabla f(x)} > 0$ & SS / LS & \(O(e^{-rt})\) for \(q=2\); \(O(t^{-q/(q-2)})\) for \(q>2\) & [D] \\
convex & UC$(q)$ & LBG, $\inf_{x\in\cC}\norm{\nabla f(x)} > 0$ & OL \(\ell/(t+\ell)\) & \(O(t^{-\ell})\) for \(q=2\); \(O(t^{-\ell+\varepsilon}+t^{-q/(q-2)})\) for \(q>2\) & [E] \\
\addlinespace[0.2em]
\multicolumn{6}{@{}l}{\emph{\textbf{Curved sets without location information}}}\\
HEB\((\theta)\) & UC$(q)$ & unrestricted & SS / LS & \(O(t^{-1/(1-2\theta/q)})\) & [F] \\
HEB\((\theta)\) & UC$(q)$ & unrestricted & OL \(\ell/(t+\ell)\) & \(O(t^{-\ell+\varepsilon}+t^{-1/(1-2\theta/q)})\) & [G] \\
SC & SC & unrestricted & SS / LS & \(O(1/t^2)\) & [H] \\
SC & SC & unrestricted & SS / LS & \(\Omega(1/t^2)\) & [I] \\
SC & SC & unrestricted & deterministic FO-LMO & \(\Omega(1/t^2)\) & [J] \\
SC quad. & \(\ell_p\)-ball, \(p\ge 3\) & zero-gradient boundary minimizer & SS / LS & \(\Theta(t^{-p/(p-1)})\) & [K] \\
\addlinespace[0.2em]
\multicolumn{6}{@{}l}{\emph{\textbf{Wolfe's proper-face regime}}}\\
SC & polytope & $x^\star \in \relinterior(F)$, \(\dim F \ge 1\) & SS / LS & \(\Omega(1/(t\log^{2+\delta}t))\) i.o. for every \(\delta>0\) & [L] \\
HEB\((\theta)\) & polytope & $x^\star \in \relinterior(F)$, active face identified & OL \(\ell/(t+\ell)\) & \(O(t^{-1/(1-\theta)})\) & [M] \\
\addlinespace[0.2em]
\multicolumn{6}{@{}l}{\emph{\textbf{Beyond-\(1/t\) via local dual sharpness (this paper)}}}\\
convex & convex & LDS, unrestricted & SS / LS & pointwise \(o(1/t)\) &  \\
convex & convex & LDS$(q)$, unrestricted & OL \(\ell/(t+\ell)\), \(\ell\ge2\) & pointwise \(o(1/t)\) &  \\
HEB\((\theta)\) & convex & LDS$(q)$, unrestricted & SS / LS & \(O(t^{-1/(1-2\theta/q)})\) &  \\
HEB\((\theta)\) & convex & LDS$(q)$, unrestricted & OL \(2/(t+2)\) & \(O(t^{-1/(1-2\theta/q)})\) & \\
\bottomrule
\end{tabularx}
\caption{Known convergence-rate landscape for vanilla Frank--Wolfe. All objective classes in the \(f\)-column are assumed convex and smooth, and all feasible sets in the \(\cC\)-column are assumed compact; ``SC'' means strongly convex, ``SC quad.'' means strongly convex quadratic \(f\), UC$(q)$ means power-type \(q\)-uniform convexity, HEB\((\theta)\) denotes a (possibly local) H\"olderian error bound with exponent \(0<\theta\le 1/2\) (manuscript uses \(r=1/\theta\)), and LDS denotes local dual sharpness, with LDS$(q)$ indicating the corresponding power \(q\). SS is the global short-step rule, LS is exact line search, OL denotes the displayed open-loop rule or family, and FO-LMO denotes a deterministic first-order method with one linear-minimization-oracle call per iteration. The LDS little-$o$ rows are fixed-instance asymptotic statements. Rows [B], [I], and [J] are uniform worst-case or finite-horizon lower bounds; the remaining rows use the quantifiers stated in their entries and the accompanying discussion. For respective references see keys in \Cref{rem:current-sota-long}.}
\IfNotNeuripsBuild{\label{tab:known-fw-rates}}
\label{tab:known-fw-rates-long}
\end{table}

\subsection{Related Work and State of the Art\IfNeuripsBuild{ (Comprehensive Version)}}
\IfNotNeuripsBuild{\label{rem:current-sota}}
\label{rem:current-sota-long}

\Cref{tab:known-fw-rates-long} summarizes the vanilla Frank--Wolfe rate landscape most
relevant to the present paper. The organizing principle is the mechanism that
either improves or obstructs the classical \(O(1/t)\) rate. We also separate
large-scale oracle lower bounds from smooth low-dimensional witnesses whenever
both are available.

\paragraph{Classical compact-convex regime.}
On a general compact convex set, vanilla Frank--Wolfe with exact line search,
global short steps, or the classical open-loop rule has the familiar
\(O(1/t)\) primal-gap guarantee [A]. The line-search analysis goes back to
\citet{frank1956algorithm}, while \citet{levitin1966constrained} and
\citet{jaggi2013revisiting} give the modern smooth-convex formulation. This
baseline is sharp in general: simplex-type examples in \cite{jaggi2013revisiting} already show an
\(\Omega(1/t)\) obstruction in the low-iteration regime, and
\citet{lan2013complexity} extends the lower-bound picture to deterministic
first-order methods that access the feasible region through one LMO call per
iteration [B].

\paragraph{Location-based acceleration.}
A first way to beat \(1/t\) is to use information about the position of the
minimizer. If the minimizer lies in the interior of the feasible region, then
short step and exact line search become linear [C], in a line of work going back
to \citet[\S 8]{wolfe70} and \citet{gm86}. A second mechanism is the
lower-bounded-gradient (LBG) regime on curved sets. When
\(\inf_{x\in\cC}\norm{\nabla f(x)}>0\) and \(\cC\) is \(q\)-uniformly convex,
the Frank--Wolfe gap controls the displacement strongly enough to give linear
convergence for \(q=2\) and faster-than-\(1/t\) polynomial decay for \(q>2\).
This picture is classical for strongly convex sets
\citep{levitin1966constrained,demyanov1970approximate,dunn1979rates}, is
presented in modern form by \citet{garber2015faster}, and is extended to
uniformly convex sets and affine-invariant growth frameworks by
\citet{kerdreux2021projection,pena2023affine} [D]. For open-loop rules
\(\gamma_t=\ell/(t+\ell)\), the same mechanism also yields accelerated rates;
the current state of the art is the affine-invariant treatment of
\citet{WPP2023}, building on the earlier \(\eta_t=4/(t+4)\) analysis of
\citet{WKP2022} [E]. Note that rows [C] and [D] are not exhaustive, as they leave out the boundary case in which the unconstrained minimizer lies exactly on $\partial\cC$, so that $\nabla f(x^\star)=0$. In fact, it is this boundary zero-gradient regime that is used by the hard SC/SC instances in [I] and [J].

\paragraph{Curved sets without location information.}
A different family of results does not assume either an interior minimizer or a
gradient bounded away from zero. Here the improvement comes from combining
curvature of the feasible region with growth on the objective. In particular,
\citet{kerdreux2021projection} [F] and \citet{pena2023affine} [G] show that a
H\"olderian error bound together with uniform convexity of the set yields
explicit polynomial rates for short step, line search, and open loop.
Strongly convex objectives over strongly convex sets fit into this picture and
give the familiar \(O(1/t^2)\) upper bound of \citet{garber2015faster} [H]. Recent
lower bounds show that this benchmark is essentially sharp in two complementary
senses: \citet{halbey2026lower} [I] obtain matching smooth small-scale lower bounds
on Euclidean balls and ellipsoids, while \citet{grimmer2026uniform} [J] prove
large-scale lower bounds for deterministic LMO methods. In our notation, the
SC/SC regime is a special case of LDS\((2)\) + HEB\((1/2)\) (equivalently
\(r=2\)), so [I] and [J] match the \(t^{-2}\) exponent obtained by specializing
\Cref{thm:local-gap-heb-rate} in worst-case complexity, although [I] allows the
hard initialization and [J] the hard problem instance and dimension to depend
on the prescribed iteration horizon. At the low-dimensional end, the explicit quadratic over an \(\ell_p\)-ball of
\citet{zimmer2026agentic} [K] exhibits the slower exponent
\(\Theta(t^{-p/(p-1)})\) for short step and exact line search at a
zero-gradient boundary minimizer.

\paragraph{Wolfe's proper-face regime.}
When the minimizer lies in the relative interior of a proper face of a
polytope, the behavior changes again. Wolfe already observed that short step
and exact line search can become much slower in that regime and \citet{Canon_FWbound68} later refined this observation. A precise formulation
of their result is given in \citet[Theorem~2.8]{braun2025cgm}: once an iterate
\(x_{t_0}\notin F\) with
\(f(x_{t_0})<\min_{v\in\operatorname{Vert}(F)}f(v)\) has been reached, for every
\(\delta>0\), the primal gap is \(\Omega(1/(t\log^{2+\delta}t))\) for infinitely
many \(t\) [L]. Consequently, no \(O(1/(t\log^{2+\eta}t))\) guarantee, with
\(\eta>0\), can hold uniformly over this proper-face regime.
Open-loop rules behave differently here: once the active face has been
identified, the affine-invariant weak-growth framework of \citet{WPP2023} [M]
gives \(O(t^{-1/(1-\theta)})\) rates, and for strongly convex objectives this
recovers an \(O(t^{-2})\) decay.

All results in this related work section concern the vanilla Frank--Wolfe algorithm of \citep{frank1956algorithm,levitin1966constrained}. Restarted \citep{kerdreux2018restarting},
away-step/pairwise \citep{lacoste2015global}, blended variants \citep{braun2019blended,tsuji2022pairwise} or otherwise modified variants can exploit additional structural assumptions and yield further acceleration, but those methods are complementary
to the present discussion.
 }

\subsection{Auxiliary Results}

The following inequality is standard; see e.g., \citep{kerdreux2021projection,braun2025cgm}.

\begin{lemma}[Uniformly-convex geometric gap bound\leanverified]\label{lem:uc-gap}
Let $\cC$ be $(\alpha,q)$-uniformly convex, let $x\in\cC$, and let
$s\in\argmin_{y\in\cC}\innp{g,y}$.
Then
\begin{equation}
  \frac{\alpha}{2}\norm{g}\,\norm{x-s}^q
  \le
  \innp{g, x-s}.
\end{equation}
\end{lemma}

\begin{proof}
  \label{app:proof-uc-gap}
By optimality of $s$,
\[
  \innp{g, y-s} \ge 0
  \qquad \forall y\in\cC.
\]
If $g=0$, there is nothing to prove.
Otherwise, set $z=-g/\norm{g}$ and use \Cref{def:uc} at $\lambda=1/2$.
The point
\[
  y \defeq \frac{x+s}{2} + \frac{\alpha}{4}\norm{x-s}^q z
\]
belongs to $\cC$, hence
\[
  0
  \le
  \innp{g,y-s}
  =
  \frac12 \innp{g,x-s}
  -
  \frac{\alpha}{4}\norm{g}\,\norm{x-s}^q.
\]
Rearranging gives the claim.
\end{proof}

\subsection{LDS vs. uniform convexity}
\label{sec:lds-vs-uc}

\begin{proposition}[A uniformly-convex patch with residual gap yields local dual sharpness\leanverified]
\label{prop:patch-local-gap}
Let \(M\subseteq\cC\) and \(K\subseteq\RR^d\).
Assume that \(K\) is \((\alpha,q)\)-uniformly convex for some \(q\in\RR\) with \(q\ge 2\).
Suppose there exist \(\rho,\beta>0\) such that whenever \(x\in\cC\) satisfies
\[
  \dist(x,M) < \rho,
\]
then \(x\in K\), and for every \(g\in\RR^d\), if \(s\) is the atom selected by
the LMO for the linear objective \(y\mapsto \innp{g,y}\), then at least
one of the following holds:
\begin{enumerate}[leftmargin=2em]
  \item \(s\in K\) and \(s\) also supports \(K\) in direction \(g\), namely
  \[
    \innp{g,y-s} \ge 0
    \qquad \forall y\in K;
  \]
  \item
  \[
    \beta \norm{g} \le \innp{g,x-s}.
  \]
\end{enumerate}
Then the LMO satisfies local dual sharpness around \(M\)
with constants
\[
  \left(
    \min\left\{
      \frac{\alpha}{2},
      \frac{\beta}{\max\{1,{\bigl(\mathrm{diam}(\cC)\bigr)}^q\}}
    \right\},
    q
  \right).
\]
\end{proposition}

\begin{proof}
  \label{app:proof-patch-local-gap}
Fix \(x\in\cC\) with \(\dist(x,M)<\rho\), fix \(g\in\RR^d\), and let \(s\) be
the atom selected by the LMO.
If \(g=0\), then both sides of the desired inequality vanish, so there is
nothing to prove.

Assume first that alternative (i) holds.
Then \(x\in K\) by hypothesis, \(s\in K\), and \(s\) supports \(K\) in
direction \(g\).
Applying \Cref{lem:uc-gap} to the uniformly convex set \(K\) yields
\[
  \frac{\alpha}{2}\norm{g}\,\norm{x-s}^q
  \le
  \innp{g,x-s}.
\]

Assume next that alternative (ii) holds.
Because both \(x\) and \(s\) belong to \(\cC\), one has
\[
  D_\cC \defeq \mathrm{diam}(\cC),
  \qquad
  \norm{x-s}\le D_\cC,
\]
and since \(r\mapsto r^q\) is increasing on the nonnegative reals,
\[
  \norm{x-s}^q \le D_\cC^q \le \max\{1,D_\cC^q\}.
\]
Therefore
\[
  \frac{\beta}{\max\{1,D_\cC^q\}}
  \norm{g}\,\norm{x-s}^q
  \le
  \beta \norm{g}
  \le
  \innp{g,x-s}.
\]

Taking the minimum of the two coefficients gives
\[
  \min\left\{
    \frac{\alpha}{2},
    \frac{\beta}{\max\{1,D_\cC^q\}}
  \right\}
  \norm{g}\,\norm{x-s}^q
  \le
  \innp{g,x-s},
\]
which is exactly the claimed local dual sharpness bound.
\end{proof}

\begin{example}[Examples beyond globally uniformly convex sets]
\label{rem:local-gap-examples}
The local dual sharpness hypothesis is genuinely weaker than global uniform
convexity.
Two explicit families in $\RR^2$ to keep in mind are:
\begin{enumerate}[leftmargin=2em]
  \item the stadium
  \[
    \cC_{\mathrm{stad}}(a)
    \defeq
    \bigl([-a,a]\times\{0\}\bigr) + B_2(0,1),
    \qquad a>0,
  \]
  together with any compact \(M\) contained in the relative interior of the
  right rounded cap, equivalently of the open semicircle
  \[
    \Gamma_a^+
    \defeq
    \{(a+\cos\theta,\sin\theta): -\pi/2<\theta<\pi/2\},
    \qquad
    K_a^+ \defeq (a,0)+B_2(0,1);
  \]
  \item the truncated disk
  \[
    \cC_{\mathrm{tr}}(b)
    \defeq
    B_2(0,1)\cap\{x_1\le b\},
    \qquad 0<b<1,
  \]
  together with any compact \(M\) contained more specifically in the lower arc
  \[
    \Gamma_b^-
    \defeq
    \{(x_1,x_2)\in \partial B_2(0,1) : x_2<-\sqrt{1-b^2}\}.
  \]
\end{enumerate}
Neither family is globally uniformly convex: the stadium contains a line
segment, while the truncated disk contains a facet.
Nevertheless, in both cases the active curved patch has a Euclidean-ball
geometry, hence \((\alpha,2)\)-uniformly convex, and the complementary LMO
directions retain a strictly positive residual support gap.
So \Cref{prop:patch-local-gap} applies and yields local dual sharpness with
\(q=2\). Higher-dimensional capsules and truncated Euclidean balls satisfy the
same patch/residual-gap criterion; the planar examples are used only to
visualize the mechanism.
\end{example}

\begin{proof}[Proof and details for the local-gap examples\leanverified]
  \label{app:proof-local-gap-examples}
We check the hypotheses of \Cref{prop:patch-local-gap} separately for the two examples.

\begin{enumerate}[leftmargin=2em]
  \item \textbf{Stadium.}
  Let
  \[
    \overline{\Gamma}_a^+
    =
    K_a^+\cap \partial \cC_{\mathrm{stad}}(a)
    =
    \{(a+\cos\theta,\sin\theta): -\pi/2\le \theta\le \pi/2\}.
  \]
  Since \(M\) is a compact subset of the relative interior of
  \(\overline{\Gamma}_a^+\),
  there exists
  \(\rho>0\) such that the closed neighborhood
  \[
    U_a
    \defeq
    \{x\in \cC_{\mathrm{stad}}(a) : \dist(x,M)\le \rho\}
  \]
  is contained in \(K_a^+\) and does not meet the two endpoints
  \((a,\pm 1)\).
  The set \(K_a^+\) is a translate of the Euclidean unit disk, hence
  \((\alpha,2)\)-uniformly convex for some \(\alpha>0\).

  Fix \(x\in U_a\) and \(g\in\RR^2\). If \(g=0\), alternative~(ii) is
  immediate, so assume \(g\neq 0\) and write \(u\defeq g/\norm{g}\).
  If the selected atom \(s\) lies in \(K_a^+\), then in fact \(s\) is the
  minimizer of the linear functional \(y\mapsto \innp{u,y}\) over the disk
  \(K_a^+\), so alternative~(i) in \Cref{prop:patch-local-gap} holds.

  It remains to consider the case where the selected atom lies outside
  \(K_a^+\).
  This can only happen for directions with \(u_1\ge0\).
  To get a uniform constant, consider the compact closure
  \[
    \Sigma_a \defeq \{u\in\RR^2 : \norm{u}=1,\ u_1\ge 0\}.
  \]
  For \(u\in\Sigma_a\), the minimum of \(y\mapsto \innp{u,y}\) over the stadium
  is
  \[
    m_a(u)
    =
    \min_{y\in \cC_{\mathrm{stad}}(a)} \innp{u,y}
    =
    -a u_1 - 1,
  \]
  because \(\cC_{\mathrm{stad}}(a)=([-a,a]\times\{0\})+B_2(0,1)\).
  Hence, for every selected atom \(s\),
  \[
    \innp{g,x-s}
    =
    \norm{g}\bigl(\innp{u,x}-m_a(u)\bigr).
  \]
  Define
  \[
    \Phi_a(x,u)\defeq \innp{u,x}-m_a(u)
    \qquad (x\in U_a,\ u\in\Sigma_a).
  \]
  This function is continuous on the compact set \(U_a\times \Sigma_a\).
  If \(\Phi_a(x,u)=0\), then \(x\) is itself a minimizer of
  \(y\mapsto \innp{u,y}\) over \(\cC_{\mathrm{stad}}(a)\).
  For \(u_1>0\), every minimizer lies on the left rounded cap.
  For \(u_1=0\), every minimizer lies on one of the horizontal line segments
  together with its two endpoints.
  Since \(U_a\subseteq K_a^+\setminus\{(a,\pm 1)\}\), neither possibility can
  occur.
  Therefore \(\Phi_a>0\) on \(U_a\times \Sigma_a\), and compactness gives
  \[
    \beta_a
    \defeq
    \min_{(x,u)\in U_a\times \Sigma_a} \Phi_a(x,u)
    >
    0.
  \]
  Consequently, whenever the selected atom lies outside \(K_a^+\),
  \[
    \innp{g,x-s}
    \ge
    \beta_a \norm{g},
  \]
  so alternative~(ii) in \Cref{prop:patch-local-gap} holds as well.
  The proposition therefore yields local dual sharpness with \(q=2\)
  around \(M\).

  \item \textbf{Truncated disk.}
  Write
  \[
    h \defeq \sqrt{1-b^2},
    \qquad
    K \defeq B_2(0,1).
  \]
  Since \(M\subseteq \Gamma_b^-\) is compact, there exist \(\rho,\eta>0\) such
  that
  \[
    U_b
    \defeq
    \{x\in \cC_{\mathrm{tr}}(b) : \dist(x,M)\le \rho\}
    \subseteq
    \{x\in \RR^2 : x_2 \le -h-\eta\}.
  \]
  In particular, \(U_b\) is disjoint from the truncating facet
  \[
    F_b \defeq \{(b,t): -h\le t\le h\}.
  \]
  The set \(K=B_2(0,1)\) is again \((\alpha,2)\)-uniformly convex for some
  \(\alpha>0\).

  Fix \(x\in U_b\) and \(g\in\RR^2\). As before, the case \(g=0\) is
  immediate from alternative~(ii), so assume \(g\neq 0\) and write
  \(u\defeq g/\norm{g}\).
  If the full-disk minimizer \(-u\) satisfies \(-u_1\le b\), then
  \(-u\in \cC_{\mathrm{tr}}(b)\), so the selected atom is exactly \(-u\), which
  supports the disk \(K\) in direction \(g\).
  Thus alternative~(i) holds.

  It remains to consider directions for which \(-u_1>b\), equivalently
  \(u_1<-b\).
  For a uniform lower bound we again pass to the compact closure
  \[
    \Sigma_b \defeq \{u\in\RR^2 : \norm{u}=1,\ u_1\le -b\}.
  \]
  For every \(u\in\Sigma_b\), minimizing \(y\mapsto \innp{u,y}\) over
  \(\cC_{\mathrm{tr}}(b)\) forces the minimizer onto the facet \(F_b\), and the
  optimal value is
  \[
    m_b(u)
    =
    \min_{y\in \cC_{\mathrm{tr}}(b)} \innp{u,y}
    =
    b u_1 - h\lvert u_2\rvert.
  \]
  Hence, for every selected atom \(s\),
  \[
    \innp{g,x-s}
    =
    \norm{g}\bigl(\innp{u,x}-m_b(u)\bigr).
  \]
  Define
  \[
    \Phi_b(x,u)\defeq \innp{u,x}-m_b(u)
    \qquad (x\in U_b,\ u\in\Sigma_b).
  \]
  This function is continuous on the compact set \(U_b\times \Sigma_b\).
  If \(\Phi_b(x,u)=0\), then \(x\) is a minimizer of
  \(y\mapsto \innp{u,y}\) over \(\cC_{\mathrm{tr}}(b)\), hence \(x\in F_b\).
  But \(U_b\cap F_b=\varnothing\), a contradiction.
  Therefore \(\Phi_b>0\) on \(U_b\times \Sigma_b\), and compactness gives
  \[
    \beta_b
    \defeq
    \min_{(x,u)\in U_b\times \Sigma_b} \Phi_b(x,u)
    >
    0.
  \]
  Consequently, whenever the selected atom does not support the full disk,
  \[
    \innp{g,x-s}
    \ge
    \beta_b \norm{g},
  \]
  which is alternative~(ii).
  Applying \Cref{prop:patch-local-gap} finishes the proof.
\end{enumerate}
\end{proof}

\begin{remark}[Strict convexity alone is not enough\leanverified]
\label{rem:strict-not-enough}
The local dual sharpness hypothesis from \Cref{def:local-gap} is a \emph{power-type}
statement, so compact strict convexity by itself does not suffice.
Indeed, after restricting to a sufficiently small neighborhood and closing the
body smoothly elsewhere, one can build a compact strictly convex body in $\RR^2$ whose boundary near a support point has the superflat graph
\[
  y = e^{-1/x^2}
  \qquad (x\neq 0),
  \qquad
  y(0)=0.
\]
For the vertical normal direction, the support gap at a boundary point
\((u,e^{-1/u^2})\) is then \(e^{-1/u^2}\), while the chord length is
comparable to \(|u|\).
Hence
\[
  e^{-1/u^2} = o(|u|^q)
  \qquad \text{for every fixed } q,
\]
so no finite-exponent inequality of the form
\[
  A \norm{g}\,\norm{x-s}^q \le \innp{g,x-s}
\]
can hold near that support point.
Thus compact strict convexity is weaker than the power-type local dual
sharpness framework used here.
\end{remark}

\subsection{\texorpdfstring{An example with $o(1/t)$ convergence without local HEB}{An example with o(1/t) convergence without local HEB}}
\label{sec:superflat-short-step}

A natural question to ask is if $o(1/t)$ convergence under LDS implies some polynomial convergence rate of the form $O(t^{-1-\eta})$ for some $\eta>0$. We will now show that this is not the case by providing an instance whose primal gap is $\Theta(1/(t(\log t)^4))$ and hence is not $O(t^{-1-\eta})$ for any $\eta>0$.

Fix $a\in(0,1/2)$ and define the convex profile
\[
  \phi(u)
  \defeq
  \begin{cases}
    0, & u\le0,\\
    e^{-1/u}, & 0<u\le a,\\
    e^{-1/a}+e^{-1/a}a^{-2}(u-a), & u>a.
  \end{cases}
\]
We will now show that for the instance
$$
\min_{x\in B_2(0,1)} \phi(1-x_1),
$$
which satisfies global LDS but no local HEB of finite order, we obtain a primal-gap convergence of order $\Theta(1/(t(\log t)^4))$.

\begin{proposition}[A logarithmic improvement over $1/t$]
\label{prop:superflat-short-step}
Let $\cC=B_2(0,1)\subseteq\RR^2$, $p=e_1$, and
$f(x)\defeq\phi(1-x_1)$. Let $\widehat L>0$ be any Lipschitz constant for
$\phi'$, and run the global short-step rule with smoothness parameter
$L=\widehat L$ from $x_0=(1-u_0)e_1$, where $u_0\in(0,a)$.
Then $\cC$ satisfies global LDS with constants $(1/2,2)$, $f$ is convex and
$\widehat L$-smooth with unique constrained minimizer $p$, and $f$ satisfies
no local HEB of any finite order. Nevertheless, its short-step primal gap obeys
\[
  \lim_{t\to\infty}
  \frac{t(\log t)^4}{\widehat L}\bigl(f(x_t)-f(p)\bigr)
  =1.
\]
In particular,
\[
  f(x_t)-f(p)
  \sim
  \frac{\widehat L}{t(\log t)^4}
  =o(1/t),
\]
but this gap is not $O(t^{-1-\eta})$ for any $\eta>0$.
\end{proposition}

\begin{proof}
We first verify the basic properties of the instance. On $0<u<a$,
\[
  \phi'(u)=\frac{e^{-1/u}}{u^2},
  \qquad
  \phi''(u)=\frac{e^{-1/u}(1-2u)}{u^4}\ge0.
\]
The values and first derivatives match at $0$ and $a$; on each of the two complementary
intervals, i.e., $(-\infty,0]$ and $[a,\infty)$, the derivative is constant. Thus $\phi$ is convex and $C^1$, and $\phi'$ is
Lipschitz because $\phi''$ is bounded on $(0,a)$. As $f(x) = \phi(1-x_1)$, it follows that $f$ is also
convex and smooth. Moreover, every $x\in\cC$ has $ 1-x_1 \geq 0$, with equality only at
$p=e_1$, hence $p$ is the unique constrained minimizer.

Next, we verify LDS for our instance via a direct computation. For $g\ne0$, write $v=g/\norm{g}$; the unit-ball
LMO is uniquely $S(g)=-v$. Since $\norm{x}\le1$,
\[
  \norm{x-S(g)}^2
  =\norm{x+v}^2
  \le 2\bigl(1+\innp{v,x}\bigr),
\]
whereas
\[
  \innp{g,x-S(g)}
  =\norm{g}\bigl(1+\innp{v,x}\bigr).
\]
Hence
$\frac12\norm{g}\norm{x-S(g)}^2\le\innp{g,x-S(g)}$ globally; the case
$g=0$ is immediate.

Now, we verify that $f(x)$ does not satisfy local HEB of any order. To this end, for $x(u)=(1-u)e_1$ with $u\downarrow0$, one has
$\dist(x(u),\{p\})=u$ and $f(x(u))-f(p)=e^{-1/u}$. Therefore, for every finite
$r>0$,
\[
  \frac{\dist(x(u),\{p\})^r}{f(x(u))-f(p)}
  =u^r e^{1/u}\xrightarrow{u\downarrow0}\infty,
\]
which rules out a local HEB of finite order.

It remains to compute the trajectory. If $x_t=(1-u_t)e_1$ with
$0<u_t<a$ (as holds initially and will be preserved by the update), then we have for the gradient
$\nabla f(x_t)=-\phi'(u_t)e_1$, for the Frank--Wolfe vertex $s_t = S(\nabla f(x_t)) = S(- \phi'(u_t)e_1) = e_1$, for the Frank--Wolfe gap $g_t = \innp{\nabla f(x_t), x_t-s_t} = \innp{- \phi'(u_t)e_1, (1-u_t) e_1 - e_1} = \innp{ \phi'(u_t)e_1, u_t e_1} = \phi'(u_t) u_t$, and for the distance $d_t = \norm{x_t - s_t} = u_t$.
Since $\phi'(0)=0$ and $\phi'$ is $\widehat L$-Lipschitz,
$\phi'(u_t)\le\widehat L u_t$ (strictly for $u_t>0$ in this example).
Thus the short step step size is
$\gamma_t = \frac{g_t}{Ld_t^2} = \phi'(u_t)/(\widehat L u_t)$ and with the update
\begin{align*}
  x_{t+1} & = x_t - \gamma_t (x_t - s_t) \\
\Leftrightarrow (1-u_{t+1}) e_1 & = (1-u_{t}) e_1 - \gamma_t ((1-u_{t}) e_1 - e_1) \\
\Leftrightarrow u_{t+1} & = u_t - \gamma_t u_t,
\end{align*}
and plugging in the definition of $\gamma_t$ we obtain
\begin{equation}
  \label{eq:superflat-u-recurrence}
  u_{t+1}
  =u_t-\frac{e^{-1/u_t}}{\widehat L u_t^2}.
\end{equation}
The strict bound $\gamma_t<1$ implies $0<u_{t+1}<u_t<a$, so this trajectory form is indeed preserved. This sequence decreases to zero: if it had a positive limit $\bar u>0$, then the decrement in
\eqref{eq:superflat-u-recurrence} would converge to the positive constant
$e^{-1/\bar u} / (\widehat L\bar u^2)$, contradicting $u_t \geq 0$.

As before we are going to analyze reciprocals. To this end, set $y_t\defeq1/u_t$ and
$b_t\defeq \widehat L^{-1}e^{-y_t}y_t^3$. Then $b_t\to0$ as $u_t \to 0$ and the recurrence
becomes
\[
  y_{t+1}=\frac{y_t}{1-b_t}
  \quad \text{ and } \quad
  y_{t+1}-y_t
  =\frac{e^{-y_t}y_t^4}{\widehat L(1-b_t)},
\]
which can be seen as follows: for the first equation write
$$
\frac{1}{y_{t+1}} = u_{t+1} = u_t-\frac{e^{-1/u_t}}{\widehat L u_t^2} = \frac{1}{y_t} - \frac{e^{-y_t}y_t^2}{\widehat L} = \frac{1}{y_t} - \frac{b_t}{y_t} = \frac{1-b_t}{y_t},
$$
and for the second equation now observe
$$
y_{t+1} - y_t = \frac{y_t}{1-b_t} - y_t = \frac{y_t b_t}{1-b_t} = \frac{e^{-y_t}y_t^4}{\widehat L(1-b_t)},
$$
via plugging in the definition of $b_t$ in the numerator in the last step. In particular, $y_{t+1}-y_t \to 0$, as $e^{-y_t}y_t^4 \to 0$ for $u_t \to 0$.

Define $R(y)\defeq\widehat L e^y/y^4$. Since $F_t=e^{-y_t}$ and by the
formula for $y_{t+1}-y_t$ above,
\[
  R(y_t)=\frac{\widehat L}{y_t^4F_t}
  \quad \text{ and } \quad
  R(y_t)=\frac{1}{(1-b_t)(y_{t+1}-y_t)}
  \sim\frac{1}{y_{t+1}-y_t}.
\]
We now show that the increments of $R(y_t)$ tend to one. Differentiating gives
\[
  R'(y)=\frac{\widehat L e^y}{y^4}\left(1-\frac{4}{y}\right).
\]
For every sufficiently large $t$, the mean value theorem ensures that there exists some
$\xi_t\in(y_t,y_{t+1})$ such that
\begin{align*}
  R(y_{t+1})-R(y_t)
  &=R'(\xi_t)(y_{t+1}-y_t)\\
  &=\frac{\widehat L e^{\xi_t}}{\xi_t^4}
    \left(1-\frac{4}{\xi_t}\right)
    \frac{e^{-y_t}y_t^4}{\widehat L(1-b_t)}\\
  &=\frac{e^{\xi_t-y_t}(y_t/\xi_t)^4(1-4/\xi_t)}{1-b_t}.
\end{align*}
As $\xi_t\in(y_t,y_{t+1})$, we have
\[
  0\le\xi_t-y_t\le y_{t+1}-y_t \longrightarrow 0,
\]
and hence $e^{\xi_t-y_t}\to1$. Moreover,
\[
  0\le\frac{\xi_t}{y_t}-1
  \le\frac{y_{t+1}-y_t}{y_t}
  \longrightarrow0,
\]
so $(y_t/\xi_t)^4\to1$. Finally, $\xi_t\ge y_t\to\infty$ gives
$1-4 / \xi_t \to 1$, while $b_t\to0$ gives $(1-b_t)^{-1}\to1$. Thus every factor in the
formula above tends to one, and hence
\[
  R(y_{t+1})-R(y_t)\longrightarrow1.
\]

Telescoping from any fixed sufficiently large $T$ and then dividing by $t$ gives
\[
  \frac{R(y_t)}{t}
  =
  \frac{R(y_T)}{t}
  +\frac{t-T}{t}
   \left[
     \frac{1}{t-T}\sum_{k=T}^{t-1}\bigl(R(y_{k+1})-R(y_k)\bigr)
   \right]
  \longrightarrow1,
\]
because the bracketed Ces\`aro average tends to one and $(t-T)/t\to1$. Thus
$$
\frac{R(y_t)}{t} = \frac{\widehat{L}}{t y_t^4 F_t} \longrightarrow 1,
$$
and hence $\frac{t y_t^4 F_t}{\widehat{L}} \to 1$. It remains to show that $y_t / \log t \to 1$, in order to obtain $\frac{t(\log t)^4F_t}{\widehat L} \to 1$. To this end, more explicitly, the preceding positive ratio satisfies
$R(y_t)/t=1+o(1)$. Applying the logarithm to this ratio we have
\[
  \log \frac{R(y_t)}{t}
  =\log(1+o(1))=o(1).
\]
Using $R(y)=\widehat L e^y/y^4$ we can expand the logarithm
\[
  \log\widehat L+y_t-4\log y_t-\log t=o(1).
\]
Since $y_t\to\infty$, division by $y_t$ yields
\[
  \frac{\log t}{y_t}
  =1-4\frac{\log y_t}{y_t}
   +\frac{\log\widehat L}{y_t}
   +o\!\left(\frac{1}{y_t}\right)
  \longrightarrow1,
\]
so that $y_t/\log t \to 1$ follows and so
\[
  \frac{t(\log t)^4F_t}{\widehat L}\longrightarrow1.
\]
Finally, $t^{1+\eta} F_t \sim \widehat L t^\eta/(\log t)^4\to\infty$ for every
$\eta>0$, showing the last assertion.
\end{proof}

\Cref{fig:superflat-short-step-tail} displays the two-dimensional
level sets and selected Frank--Wolfe iterates alongside the primal tail. Since
$f(x)=\phi(1-x_1)$, the level sets are vertical sections of the ball. The slow
convergence of $y_t/\log t$ to one explains the visible finite-time separation
from the final $L/(t(\log t)^4)$ equivalent.

\begin{figure}[t]
  \centering
  \includegraphics[width=\textwidth]{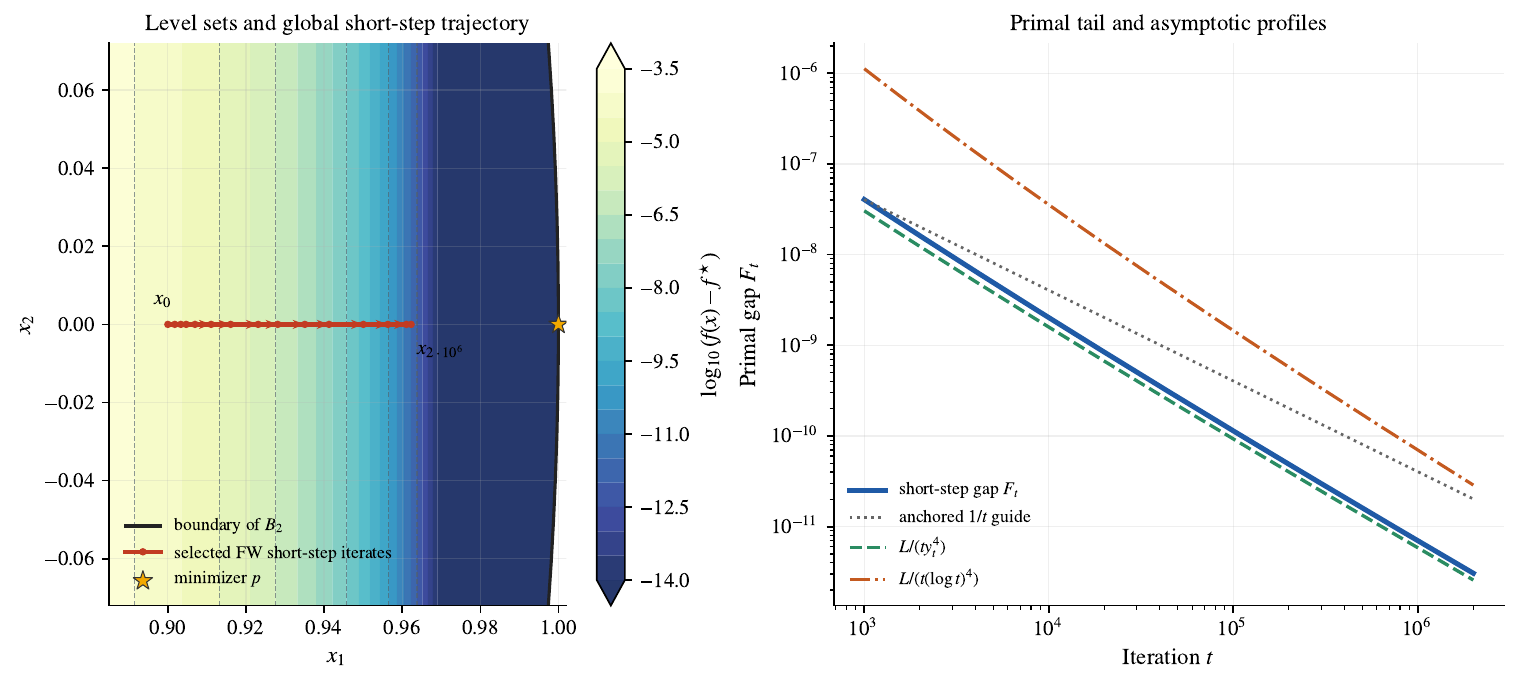}
  \caption{Fixed superflat short-step instance from
  \Cref{prop:superflat-short-step}, with $a=1/4$, $u_0=0.1$, and
  $\widehat L\approx2.55$. Left: level sets on a zoom of the right cap of
  $B_2(0,1)$ and selected iterates of global-short-step Frank--Wolfe; the star
  marks the minimizer $p=e_1$. Right: the primal gap, an anchored $1/t$ guide,
  the intermediate profile $\widehat L/(t y_t^4)$, and the final equivalent
  $\widehat L/(t(\log t)^4)$. The displayed tail runs from $t=10^3$ to
  $t=2\cdot10^6$.}
  \label{fig:superflat-short-step-tail}
\end{figure}

\begin{remark}[Dependence on step size rule]
For our objective \(f(x)=\phi(1-x_1)\), at every nonoptimal point the
gradient is a negative multiple of \(e_1\), and therefore the LMO uniquely
selects the minimizer \(p\). As there are no other minimizers on the segment
\([x_0,p]\), exact line search takes \(\gamma_0=1\), while the classical
open-loop rule has \(\gamma_0=1\) by definition. Thus both rules reach
\(x_1=p\). At the next iteration, the Frank--Wolfe gap satisfies
\[
  g_1=\innp{\nabla f(p),p-s_1}=0,
\]
independent of any LMO tie-breaking at the zero gradient. The Frank--Wolfe
gap is available in each iteration from the gradient and the already computed
LMO atom. Thus, under the standard zero-gap stopping convention, both methods
terminate at \(p\) before taking a second update; in particular, the
Frank--Wolfe algorithm with the open-loop rule does not subsequently move
away from the optimum. Hence
\Cref{prop:superflat-short-step} answers the existence question only for
global short steps; analogous examples for exact line search and classical
open loop remain open but are likely possible to construct.
\end{remark}

\subsection{An additional example with known \texorpdfstring{$q$}{q} and \texorpdfstring{$r$}{r}}
\label{sec:l4-additional-example}
In this section we provide an example for which we can explicitly compute the predicted convergence rate via \Cref{thm:local-gap-heb-rate} and compare this to the numerically obtained rates, showing that our estimation is empirically tight on long trajectories.

Consider
\[
  \cC_4\defeq\{x\in\RR^2:\norm{x}_4\le1\},
  \qquad
  f_4(x)\defeq\frac12\norm{x-e_1}_2^2.
\]
The set is uniformly convex of power type $q=4$, hence satisfies LDS$(4)$,
and strong convexity of $f_4$ gives a local HEB of order $r=2$. Therefore
\Cref{thm:local-gap-heb-rate} predicts
$F_t=O(t^{-rq/(rq-2)})=O(t^{-4/3})$ for all three step-size rules. Numerically, the fitted exponents are $1.331534$ for the short-step trajectory (and the exact-line-search trajectory, since they coincide) and $1.330033$ for classical open loop, with $R^2>0.999999$ in both cases. \Cref{fig:l4-quantitative-rates} shows the corresponding scaled tails. We stress that this is only numerical evidence on a long tail and not an asymptotic verification in the strict sense.

\begin{figure}[t]
  \centering
  \includegraphics[width=\textwidth]{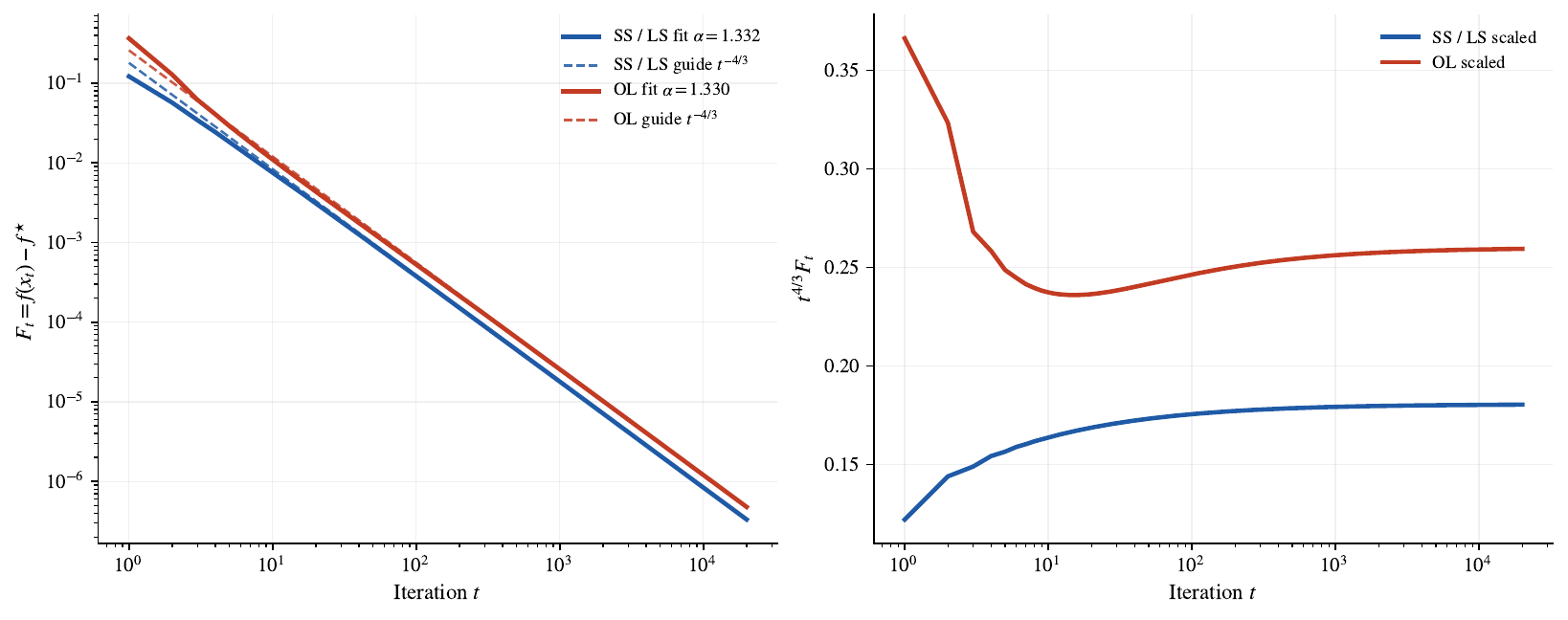}
  \caption{Quadratic $\ell_4$-ball example with known LDS power
  $q=4$ and HEB order $r=2$. Left: primal gaps and $t^{-4/3}$ guides.
  Right: the scaled quantities $t^{4/3}F_t$, which are approximately flat on
  the tested tail. The plot is a deterministic finite-horizon check of the
  predicted exponent.}
  \label{fig:l4-quantitative-rates}
\end{figure}

\subsection{Generalized Open-Loop Step-Sizes}

\begin{remark}[What changes for the fixed \(\ell/(t+\ell)\) family]
\label{rem:ol-family}
For every fixed integer \(\ell\ge 2\) and every LDS power \(q\ge2\), the same
minimizer-based open-loop argument under local dual sharpness extends to the
rule
\[
  \gamma_t = \frac{\ell}{t+\ell}.
\]
The proof uses the rescaled gaps
\[
  h_t \defeq (t+\ell)F_t
\]
instead of \((t+2)F_t\), and the corresponding coefficients become
\[
  a_t \defeq \frac{t(t+\ell+1)}{(t+\ell)^2},
  \qquad
  c_t \defeq \frac{t+\ell+1}{(t+\ell)^2}.
\]
Accordingly, the curvature and error tails acquire factors of order
\(\ell^2\). The direct shifted-bound branch and the Young-correction branch
from the classical open-loop proof persist, with the relevant small-branch and threshold
constants enlarged by factors depending only on \(\ell\).
The one-step drop along a dyadic block is therefore weaker by an
\(\ell\)-dependent constant factor, but for fixed \(\ell\) it still yields a
positive block drop, so the same contradiction argument proves
\[
  (t+\ell)\bigl(f(x_t)-f^\star\bigr)\to 0.
\]
The minimizer-set reduction from \Cref{thm:ol-set} is
unchanged, and \Cref{prop:uc-implies-local-gap} immediately yields the
uniformly-convex family specialization.
\end{remark}

\IfDeferredProofs{%

\subsection{Deferred Proofs}

\PrintDeferredProofs

\PrintDeferredProofSections
}

\IfNeuripsBuild{%
\clearpage

}
 \end{document}